\documentclass [12 pt, twoside]{amsart}
\usepackage{enumerate}

\usepackage{amsmath,amsfonts,amsthm,amssymb}
\usepackage[all]{xy}

\newtheorem{thm}{Theorem}[section]
\newtheorem{prop}{Proposition}[section] 
\newtheorem{lem}{Lemma}[section] 

\theoremstyle{definition}

\newcommand{\bi} {{\beta}}

\newcommand{\ga} {{\gamma}}

\newcommand{\Ga} {{\varGamma}}

\newcommand{\ld} {{\ldots}}
\newcommand{\sm} {{\smallsetminus}}
\newcommand{\thi} {{\theta}}
\newcommand{\de} {{\delta}}

\newcommand{\si} {{\sigma}}

\newcommand{\la} {{\lambda}}

\newcommand{\el} {{\ell}}
\newcommand{\e} {{\varepsilon}}
\newcommand{\f} {{\varphi}}
\newcommand{\mi} {{\mu}}

\newcommand{\dis}{\displaystyle}
\newcommand{\ssum}{\sum\limits}

\newcommand{\ct}{{\mathcal{T}}}

\newcommand{\cu}{{\mathcal{U}}}

\newcommand{\ch}{{\mathcal{H}}}

\newcommand{\cn}{{\mathcal{N}}}
\newcommand{\cde}{{\mathcal{D}}}

\newcommand{\ra}{{\rightarrow}}
\newcommand{\fa}{{\forall}}

\newcommand{\qb}{$\quad\blacksquare$}
\def\1{\it1\hspace*{-0.150cm}{\footnotesize{I}}}

\def\R{{\mathbb{R}}}
\def\C{{\mathbb{C}}}
\def\Q{{\mathbb{Q}}}
\def\N{{\mathbb{N}}}

\numberwithin{equation}{section}

\begin{document}

\title[ Common hypercyclic vectors ]{ Existence of common hypercyclic vectors for translation operators}

\author[Nikos Tsirivas]{Nikos Tsirivas}
\address{University College Dublin, School of Mathematical Sicences, Belfield, Dublin 4, Dublin, Ireland. Current Adress: Department of Mathematics and Applied Mathematics, University of Crete, GR-700 13 Heraklion, Crete, Greece}
\email{tsirivas@uoc.gr}

\thanks{The author was supported by Science Foundation Ireland under Grant
09/RFP/MTH2149 }


\date{}


\begin{abstract}
Let $\ch(\C)$ be the set of entire functions endowed with the topology $\ct_u$ of local uniform
convergence. Fix a sequence of non-zero complex numbers $(\la_n)$ with $|\la_n|\to +\infty$ and
$|\la_{n+1}|/|\la_n|\to 1$. We prove that there exists a residual set $G\subset \ch(\C)$ such that for every
$f\in G$ and every non-zero complex number $a$ the set $\{ f(z+\la_na):n=1,2,\ldots \}$ is dense in
$(\ch(\C),\ct_u)$. This provides a very strong extension of a theorem due to Costakis and Sambarino in \cite{6}.
Actually, in \cite{6} the above result is proved only for the case $\la_n=n$. Our result is in a sense best
possible, since there exist sequences $( \la_n )$, with $|\la_{n+1}|/|\la_n| \to l$ for some $l>1$, for which
the above result fails to hold \cite{8}.

\end{abstract}

\maketitle

\section{Introduction}\label{sec1}
\noindent Let us first fix some standard notation and terminology. Throughout this paper, we denote $\N=$
$\{1,2,\ld\}$, $\Q$, $\mathbb{R}$, $\mathbb{C}$ for the sets of natural, rational, real and complex numbers
respectively. By $\ch(\C)$ we denote the set of entire functions endowed with the topology $\ct_u$ of local
uniform convergence. For a subset $A$ of $\ch(\C)$ the symbol $\overline{A}$ denotes the closure of $A$ with
respect to the topology $\ct_u$. Let $X$ be a topological vector space. A subset $G$ of $X$ is called
$G_{\de}$ if it can be written as a countable intersection of open sets in $X$ and a subset $Y$ of $X$ is called
residual if it contains a $G_{\de}$ and dense subset of $X$.

A classical result of Birkhoff \cite{4}, which goes back to 1929, says that there exist entire functions the
integer translates of which are dense in the space of all entire functions endowed with the topology $\ct_u$ of
local uniform convergence (see also Luh \cite{11} for a more general statement). Birkhoff's proof was
constructive. Much later, during 80's, Gethner and Shapiro \cite{9} and independently Grosse-Erdmann
\cite{Grosse1} showed that Birkhoff's result can be recovered as a particular case of a much more general
theorem, through the use of Baire's category theorem. This approach, simplified substantially Birkhoff's
argument and in addition gave us precise information on the topological size of these functions. In particular,
Grosse-Erdmann proved that for every fixed sequence of complex numbers $(w_n)$ with $w_n\to \infty$, the set
$$
\Big\{ f\in \ch(\C)| \,\, \overline{\{ f(z+w_n):\;n\in\N\}} = \ch(\C) \Big\}
$$
is  $G_\de$ and dense in $\ch(\C)$, and hence ``big" in the topological sense.

Let us now rephrase the above results using the modern language of hypercyclicity. Let $(T_n:X\ra X)$ be a
sequence of continuous linear operators on a topological vector space $X$. For $x\in X$ the set $Orb(\{ T_n\}
,x):= \{ T_n(x): n=1,2,\ldots \}$ is called the \textit{orbit} of $x$ under $(T_n)$. If $(T_n(x))_{n\ge1}$ is
dense in $X$ for some $x\in X$, then $x$ is called \textit{hypercyclic} for $(T_n)$ and we say that $(T_n)$ is
\textit{hypercyclic} \cite{2}, \cite{10}. The symbol $HC(\{ T_n\})$ stands for the collection of all hypercyclic
vectors for $(T_n)$. In the case where the sequence $(T_n)$ comes from the iterates of a single operator $T:X\to
X$, i.e. $T_n:=T^n$, then we simply say that $T$ is \textit{hypercyclic} and $x$ is \textit{hypercyclic} for
$T$. If $T:X\to X$ is hypercyclic then the symbol $HC(T)$ stands for the collection of all hypercyclic vectors
for $T$. Following standard terminology, for an operator $T$ on $X$ the set $Orb(T,x):=\{ x,T(x), T^2(x), \ldots
\}$ is called the \textit{orbit} of $x$ under $T$. A simple consequence of Baire's category theorem is that for
every continuous linear operator $T$ on a separable topological vector space $X$, if $HC(T)$ is non-empty then
it is necessarily ($G_{\de}$ and) dense. For an account of results on the subject of hypercyclicity we refer to
the recent books \cite{2}, \cite{10}, see also the very influential survey article \cite{Grosse2}.

For every $a\in \mathbb{C} \setminus \{ 0\}$ consider the translation operator $T_a:\ch(\C) \to \ch(\C)$ defined
by $$T_a(f)(z)=f(z+a), \,\,\, f\in \ch(\C).$$ Thus, for $a=1$ Birkhoff's result says that $T_1$ is hypercyclic.
We note that the choice $a=1$ is not significant. The same proof works nicely for every $a\in \mathbb{C}
\setminus \{ 0\}$, that is, for such $a$, $T_a$ is hypercyclic and hence $HC(T_a)$ is $G_{\de}$ and dense in
$\ch(\C)$.

Recently, Costakis and Sambarino \cite{6} established a notable strengthening of Birkhoff's result. Namely, they
showed that, for almost all entire functions $f$, in the sense of Baire category, the set of the translates of
$f$ with respect to $na$, $n\in\N$, is dense in the space of all entire functions for every non-zero complex
number $a$. The significant new element here is the uncountable range of $a$. In the language of hypercyclicity
their result takes the following form: the family $\{ T_a| \,\, a\in \mathbb{C} \setminus \{ 0\} \}$ has a
residual set of common hypercyclic vectors i.e.,
$$ \textrm{ the set} \,\, \bigcap_{a\in \mathbb{C}\setminus \{ 0\}} HC(T_a) \,\,\, \textrm{is residual}
\,\,\,\textrm{in}\,\,\,\ch(\C),$$ or equivalently, the set
$$
\bigcap_{a\in \mathbb{C}\setminus \{ 0\} } \Big\{ f\in \ch(\C)| \,\, \overline{\{ f(z+na):\;n\in\N\}} = \ch(\C)
\Big\}
$$
is residual in $\ch(\C)$. In particular, it is non-empty.

Subsequently, Costakis \cite{7} asked whether, in this result, $n$ can be replaced by more general sequences
$(\la_n)$ of non-zero complex numbers.\\

{\bf Question}\cite{7}. Fix a sequence $(\la_n)$ of non-zero complex numbers such that $|\la_n|\ra\infty$. Are
there entire functions $f$ such that, \textit{for all} $a\in\C\sm\{0\}$, the set $\{f(z+\la_na):\;n\in\N\}$ is
dense in the space
of all entire functions?\\

In this direction Costakis \cite{7} showed that, if the sequence $(\la_n)$ satisfies a certain condition, then
the desired conclusion holds if we restrict attention to $a\in C(0,1):=\{z\in\C/|z|=1\}$. The precise condition
is that, for every $M>0$, there exists a subsequence $(\la_{n_k})$ of $(\la_n)$ such that

(i) $|\la_{n_{k+1}}|-|\la_{n_k}|>M$ for every $k=1,2,\ld$ and

(ii) $\ssum^\infty_{k=1}\dfrac{1}{|\la_{n_k}|}=+\infty$.

Obviously, sequences of the form $\la_n=bn+c$, where $b,c\in\C$, $b\neq0$, $\la_n= n(\log n)^p$, where $0<p\leq
1$ or $\la_n =n\log n \log \log n$, etc., satisfy the above condition. On the other hand, the case where the
sequence $\la_n$ is sparse, say $n^2$, is left open, since in this case condition $(ii)$ is not satisfied. And
actually this is not accidental; it reflects the limitation of the method developed in \cite{7}. This drawback
is due to a specific ``one-dimensional partition" that the author chooses. Here we overcome this difficulty, by
constructing a ``two dimensional" partition, which turns out to be the right one in order to handle sequences
$(\la_n)$ where the corresponding series in condition $(ii)$ converges. The purpose of this paper is to give a
positive answer for general $a\in\C\sm\{0\}$ that applies to a wide family of sequences $(\la_n)$. In
particular, our main result, Theorem \ref{thm1}, covers the case where $(\la_n)$ is of the form $(p(n))$, and
$p$ is any non-constant complex polynomial, as well as the case where $\la_n=e^{n^b}$ for $0<b<1$; hence for
every $0<b<1$ we have
$$
\bigcap_{a\in \mathbb{C}\setminus \{ 0\} }\Big\{ f\in \ch(\C)| \,\, \overline{\{ f(z+e^{n^b}a):\;n\in\N\}} =
\ch(\C)\Big\} \neq \emptyset.
$$
We would like to stress that the allowed growth $e^{n^b}$, $0<b<1$ in the previously mentioned example is in a
sense optimal, since the answer to the above question is negative if $(\la_n)$ grows exponentially, \cite{8},
that is,
$$
\bigcap_{a\in C(0,1) }\Big\{ f\in \ch(\C)| \,\, \overline{\{ f(z+e^na):\;n\in\N\}} = \ch(\C)\Big\} =\emptyset.
$$
So some restriction on the nature of $(\la_n)$ is clearly necessary.

Let $(\la_n)$ be a sequence of non-zero complex numbers. We adjust a non-negative real number to the sequence
$\Lambda:=(\la_n)$, the following:
\begin{align*}
i(\Lambda ):=\inf\bigg\{&a\in\R \cup \{ +\infty \} |\;\text{there exists a subsequence}\; (\mi_n)\;\text{of}\; (\la_n)\;\text{such that}\\
&a=\underset{n\ra+\infty}{\lim\sup}\bigg|\frac{\mi_{n+1}}{\mi_n}\bigg|\bigg\}.
\end{align*}
Of course $i(\Lambda )\in[0,+\infty]$. If $\la_n\ra\infty$ as $n\ra+\infty$ then $i(\Lambda )\in[1,+\infty]$.
Our main result is the following
\begin{thm}\label{thm1}
Let $\Lambda:=(\la_n)$ be a fixed sequence of non-zero complex numbers such that $\la_n\ra\infty$ as
$n\ra+\infty$ and $i(\Lambda )=1$. Then, the set
\[
\bigcap_{a\in\C\sm\{0\}}HC(\{T_{\la_na}\} ) \ \ \text{is a} \ \ G_\de, \ \ \text{dense subset of} \ \
(\ch(\C),\ct_u).
\]
In particular, there exists $f\in \ch(\C)$ such that for every $a\in \mathbb{C}\setminus \{ 0\}$
$$\overline{ \{ f(z+\la_na)| \,\, n=1,2,\ldots \} }=  \ch(\C) .$$
\end{thm}

All the work in this article is become in order to prove Theorem \ref{thm1}.\\

It seems to be an appropriate place here to comment on the ideas developed in \cite{6}, \cite{7} and to compare
them with our approach. Costakis and Sambarino's result mentioned above consists of two steps. The first one, is
to show that the set $\bigcap_{a\in C(0,1)} HC(T_a)$ is residual in $\ch(\C)$. This is accomplished by choosing
a suitable partition of the unit circle $C(0,1)$ and then an application of Runge's theorem on specific compact
sets depending on the partition concludes the argument. We stress that what we just said is a very rough idea of
their proof. In the second step they show that for any fixed $\theta \in \mathbb{R}$, $
HC(T_{e^{i\theta}})=HC(T_{re^{i\theta}})$ for every $r>0$. The proof of the latter is based on two important
results: the minimality of the irrational rotation, see for instance \cite{Ellis}, and Ansari's theorem
\cite{Ans}, which says that if $T$ is hypercyclic then for every $n\in \mathbb{N}$, $T^n$ is hypercyclic and in
addition $HC(T)=HC(T^n)$. One key element to prove Ansari's theorem is that the orbit $Orb(T,x)$ has a semigroup
structure, that is, if $T^n(x), T^m(x) \in Orb(T,x)$ then $T^n\circ T^m(x)\in Orb(T,x)$. Some nice extensions of
Ansari's theorem even in a non-linear setting can be found in \cite{Math}, \cite{Shka}, where still the
semigroup structure property plays important role in the proofs. Observe now that in our case, say $\Lambda
:=(\la_n)$, $\la_n \to \infty$ and assume for simplicity $\la_n \in \mathbb{N}$, the semigroup structure of the
orbit breaks down. The very simple reason for this ``unpleasant" situation is that we now consider parts of the
full orbit $\{ f(z+an): n=1,2,\ldots \}$, which may be very sparse. For instance, consider the sequence
$\la_n=n^2$ (for which Theorem \ref{thm1} holds). Clearly for $a\in \mathbb{C}\setminus \{ 0\}$, $f\in \ch(\C)$,
we have $T_{m^2a}\circ T_{l^2a}(f)\notin Orb( \{ T_{n^2a} \} ,f)$ in general. In view if this obstacle, we led
to follow a different approach and therefore we tried to concentrate on the first step in Costakis and
Sambarino's approach. Now the problem is how to find a suitable partition, not only for the set $C(0,1)$, which
is quite "thin", but for any given bounded sector $S$. So our main task is: for a given sequence $(\la_n)$
satisfying the hypothesis of Theorem \ref{thm1}, and a given bounded sector $S\subset \mathbb{C}$ find a
suitable partition of $S$ in order to show that the set $\bigcap_{a\in S}HC(\{T_{\la_na}\} )$ is $G_{\de}$ and
dense in $\ch(\C)$. Then, covering the complex plane by countable many such sectors and applying Baire's
category theorem we will be done. We mention that the second step of Costakis and Sambarino's result can be also
obtained as a particular case of a general result due to Conejero, M\"{u}ller and Peris \cite{5} concerning
hypercyclic $C_0$ semigroups, see also \cite{2}.

There is a fast growing literature on the subject of common hypercyclic vectors for certain uncountable families
of sequences of operators. For instance, Bayart and Matheron \cite{3}, answering a question from \cite{Cost},
they show, among other things, the existence of entire functions $f$ such that for every non-negative real
number $s\geq 0$ and for every $a\in \mathbb{C} \setminus \{ 0\}$, $\overline{ \{ n^sf(z+na): n=1,2,\ldots \}
}=\ch(\C)$. Shkarin in \cite{13}, extending the Costakis and Sambarino's result above, proves the following: the
set $\bigcap_{a,b\in \mathbb{C} \setminus \{ 0\} } HC ( bT_a)$ is residual in $\ch(\C)$. There are also several
results concerning the existence of common hypercyclic vectors for other type of operators such as, weighted
shifts, adjoints of multiplication operators, differentiation and composition operators; see for instance,
\cite{AbGo}, \cite{Ba1}-\cite{Ber}, \cite{ChaSa1}-\cite{GaPa}, \cite{10}, \cite{LeMu}, \cite{Math},
\cite{San}-\cite{13}. There are also results going to the opposite direction, namely the non-existence of common
hypercyclic vectors for certain families of operators, see \cite{Ba3}, \cite{BaGr1}, \cite{8}, \cite{13}. A most
worthy and very general result, due to Shkarin \cite{13}, is the following. For any given linear and continuous
operator $T$ acting on a complex topological vector space with non-trivial dual, the family $\{ rT+aI: r>0, a\in
\mathbb{C} \}$ does not have a common hypercyclic vector.

The paper is organized as follows. Sections $2-7$ occupy the proof of Theorem \ref{thm1}. In the last section,
section $8$, we give some illustrating examples of sequences $(\la_n)$ satisfying the hypothesis of Theorem
\ref{thm1}, which fall into four distinct classes.

\section{A special case of Theorem \ref{thm1}: an outline of the proof and notation}\label{sec2}
\noindent

In this section we provide a general framework for attacking our problem, by considering a particular case of
the sequence $(\la_n)$. It turns out that handling this case is actually all what we need in order to establish
our main result, namely Theorem \ref{thm1}. This reduction is explained and presented in full detail in section 7. Let us
now describe the extra properties we impose on the sequence $(\la_n)$.

Let $(\la_n)$ be a sequence of non-zero complex numbers satisfying the following:
\begin{enumerate}
\item[1)] $|\la_{n+1}|-|\la_n|\ra+\infty$ as $n\ra+\infty$
\item[2)] $\dfrac{\la_{n+1}}{\la_n}\ra1$ as $n\ra+\infty$
\item[3)] $\underset{n\ra+\infty}{\dis\lim\inf}\Big(n\Big(\Big|\dfrac{\la_{n+1}}{\la_n}\Big|
    -1\Big)\Big)>0$.
\end{enumerate}
A sample of sequences satisfying the above three properties is: $\la_n=n^c$, $c>1$, $\la_n=n^{\beta}\log n$,
$\beta \geq 1$, $\la_n=n^{\gamma}/\log (n+1)$, $\gamma >2$, etc. Our main task is to prove the following special
case of Theorem \ref{thm1}.
\begin{thm}\label{thm2}
Fix a sequence $(\lambda_n)$ of non-zero complex numbers which satisfies the above properties 1), 2), 3). Then
$\bigcap\limits_{a\in \mathbb{C}\setminus \{ 0\} }HC(\{ T_{\la_na}\} )$ is a $G_\de$ and dense subset of
$(\ch(\C),\ct_u)$.
\end{thm}

Let us now describe the steps for the proof of Theorem \ref{thm2}. Consider the sectors
\[
S^k_n:=\bigg\{a\in\C|\exists\;r\in\bigg[\frac{1}{n},n\bigg]\;\text{and}\;t\in\bigg[\frac{k}{4},
\frac{k+1}{4}\bigg]\;\text{such that}\;a=re^{2\pi it}\bigg\}
\]
for $k=0,1,2,3$ and $n=2,3,\ld\,$. Since
\[
\bigcap_{a\in\C\sm\{0\}}HC(\{ T_{\la_na} \} )=\bigcap^3_{k=0}\bigcap^\infty_{n=2} \bigcap_{a\in S^k_n}HC(\{
T_{\la_na}\} ),
\]
an appeal of Baire's category theorem reduces the proof of Theorem \ref{thm2} to the following

\begin{prop}\label{prop1}
Fix a sequence $(\lambda_n)$ of non-zero complex numbers which satisfies the above properties 1), 2), 3). Fix
four real numbers $r_0,R_0,\thi_0,\thi_T$ such that $0<r_0<1<R_0<+\infty$, $0\le\thi_0<\thi_T \leq 1$,
$\thi_T-\thi_0=\dfrac{1}{4}$ and consider the sector $S$ defined by
\[
S:=\{a\in\C | \,\,\text{there exist} \,\,r\in[r_0,R_0] \ \ \text{and} \ \ t\in[\thi_0,\thi_T] \ \ \text{such
that} \ \ a=re^{2\pi it}\}.
\]
Then $\bigcap\limits_{a\in S}HC(\{ T_{\la_na}\} )$ is a $G_\de$ and dense subset of $(\ch(\C),\ct_u)$.
\end{prop}
For the proof of Proposition \ref{prop1} we introduce some notation which will be carried out throughout this
paper. Let $(p_j)$, $j=1,2,\ld$ be a dense sequence of $(\ch(\C),\ct_u)$, (for instance, all the polynomials in
one complex variable with coefficients in $\Q+i\Q$). For every $m,j,s,k\in\N$ we consider the set $$E(m,j,s,k):=
\!\Big\{f\!\in\!\ch(\C)\,\, | \forall a\!\in\! S \,\,\,\exists  n\!\in\!\N, n\!\le\! m : \dis\sup_{|z|\le
k}|f(z\!+\!\la_na)\!-\!p_j(z)|\!<\!\frac{1}{s}\Big\}.$$
By Baire's category theorem and the three lemmas stated below, Proposition \ref{prop1} readily follows.
\begin{lem}\label{lem1}

\[
\bigcap_{a\in S}HC(\{ T_{\la_na}\} )=\bigcap^\infty_{j=1}\bigcap^\infty_{s=1}\bigcap^\infty_{k=1}
\bigcup^\infty_{m=1}E(m,j,s,k).
\]
\end{lem}

\begin{lem}\label{lem2}
For every $m,\!j,\!s,\!k\!\in\!\N$ the set $E(m,\!j,\!s,\!k)$ is open in $(\ch(\C),\ct_u)$.
\end{lem}

\begin{lem}\label{lem3}
For every $j,\!s,\!k\in\N$ the set $\bigcup\limits^\infty_{m=1}E(m,j,s,k)$ is dense in $(\ch(\C),\ct_u)$.
\end{lem}
The proof of Lemma \ref{lem2} is similar to that in Lemma 9 of \cite{6} and it is omitted. Coming now to Lemma
\ref{lem1}, the inclusion
$$
\bigcap^\infty_{j=1}\bigcap^\infty_{s=1}\bigcap^\infty_{k=1} \bigcup^\infty_{m=1}E(m,j,s,k) \subset
\bigcap_{a\in S}HC(\{ T_{\la_na}\} )
$$
is easy to establish, therefore it is left as an exercise to the interested reader. At this point we would like
to stress that Lemmas \ref{lem2}, \ref{lem3} along with the above inclusion immediately imply that the set
$\bigcap_{a\in S}HC(T_{\la_na})$ is residual, hence non-empty. However, one can get a more precise information
concerning the topological structure of the set $\bigcap_{a\in S}HC(T_{\la_na})$ which is actually $G_{\delta}$.
The proof of the last fact, which is not so obvious, is postponed till the last section, section 6. We now move
on to Lemma \ref{lem3}. This lemma is the heart of our argument and its proof is long and difficult. In order to
present it in a more digestive form, we give below a very rough sketch of the main ideas involved in the proof.
As the reader may notice, our strategy shares certain similarities with the proof of Lemma 10 from \cite{6}. On
the other hand we will indicate the points at which our argument differentiates from that in \cite{6}.

We start by fixing $j_1, s_1,k_1\in \mathbb{N}$. We need to prove that
$\bigcup\limits^\infty_{m=1}E(m,j_1,s_1,k_1)$ is dense in $(\ch(\C),\ct_u)$. For simplicity we write
$p_{j_1}=p$. Consider $g\in\ch(\C)$, a compact set $C\subset \mathbb{C}$ and $\epsilon_0>0$. We seek $f\in
\ch(\C)$ and a positive integer $m_1$ such that
\begin{eqnarray}
f\in E(m_1,j_1,s_1,k_1)  \label{eq1}
\end{eqnarray}

and

\begin{eqnarray}
\sup_{z\in C}|f(z)-g(z)|<\epsilon_0.  \label{eq2}
\end{eqnarray}

\emph{ What is done in Lemma 10 from  \cite{6}}: The authors in \cite{6},\cite{7} deal with the unit circle
instead of the sector $S$. Then they define a suitable one dimensional partition of the unit circle $\{ a_1,
a_2, \ldots ,a_n\}$ and choose appropriate terms $\lambda_{\mu_1}, \ldots ,\lambda_{\mu_n}$ of the sequence
$(\lambda_n)$ so that the discs
$$ B, \,\,\, B_i:=B+a_i\lambda_{\mu_i}, \,\,\, i=1,\ldots ,n$$
are pairwise disjoint, where $B$ is a closed disc centered at zero with sufficiently large radius $R$ and $R$
only depends on fixed initial conditions of the problem. Then setting
$$L:=B\bigcup  \left( \cup_{i=1}^nB_i \right) ,$$
defining a suitable holomorphic function on $L$ and using Runge's theorem they conclude the existence of a
polynomial which satisfies a finite number of the desired inequalities. Taking advantage of the fact that the
partition $\{ a_1, a_2, \ldots ,a_n\}$ is very thin, i.e. $a_i$ is close enough to $a_{i+1}$ for $i=1,\ldots
,n-1$ they are able to check the validity of the remaining inequalities for all the points of the unit circle.

\emph{ What we do}: Our argument boils down in finding a desired two dimensional partition $\{ a_1, \ldots a_n
\}$ of the above sector $S$. The construction of this partition consists of five steps and is presented in
section 3. Then we adjust a specific term $\lambda(a_j)$, $j=1,\ldots ,n$ of the sequence $(\lambda_n)$ to each
one of the above numbers $a_j$ of the partition and we define the discs
$$ B, \,\,\, B_j:=B+a_j\lambda(a_j), \,\,\, j=1,\ldots ,n$$ so that they are pairwise disjoint. Once this is
established we more or less follow the procedure mentioned above in order to prove (\ref{eq1}), (\ref{eq2}).

\subsection{Good properties of the sequence $(\lambda_n)$} \label{good}
Let $(\la_n)$ be a sequence of non-zero complex numbers satisfying the following:
\begin{enumerate}
\item[1)] $|\la_{n+1}|-|\la_n|\ra+\infty$ as $n\ra+\infty$
\item[2)] $\dfrac{\la_{n+1}}{\la_n}\ra1$ as $n\ra+\infty$
\item[3)] $\underset{n\ra+\infty}{\dis\lim\inf}\Big(n\Big(\Big|\dfrac{\la_{n+1}}{\la_n}\Big|
    -1\Big)\Big)>0$.
\end{enumerate}

Let $r_0,R_0,\thi_0,\thi_T$ be positive numbers such that $0<r_0<1<R_0<+\infty$, $0\le\thi_0<\thi_T\leq 1$. Let
also $c_0$, $c_1$, $c_2$, $c_3$, $c_4$ be positive numbers such that $c_0>2$, $c_1>2$, $0<c_2<1$, $c_3>0$,
$2c_3<\liminf_n (n(\frac{|\la_{n+1}|}{|\la_n|} -1))$ and $c_4:=\frac{r_0c_3}{2}$. Finally, define
$$m_0:=\left[\frac{R_0c_1}{r_0}\right]+1, $$
$$k_0:=\left[ \frac{2c_0}{c_2} \right] +1 ,$$ where the symbol $[x]$ stands for the integer part of a real
number $x\in \mathbb{R}$. Using elementary calculus and the above properties of $(\la_n)$ it is easy to see that
there exists a fixed natural number $n_0$ such that for every $n\in\N$, $n\ge n_0$ all the following 8
inequalities hold:
\begin{enumerate}
\item[(2.1)] $|\la_n|\cdot\ssum^{m_0-1}_{k=0}\dfrac{1}{|\la_{n+k}|}>\dfrac{R_0}{r_0}c_1$
\item[(2.2)] $|\la_{n+1}|-|\la_n|>4\dfrac{c_0}{r_0}$
\item[(2.3)]  $|\la_n|>\dfrac{4c_0}{r_0}$
\item[(2.4)]  $|\la_n|\cdot\ssum^{k_0}_{i=1}\dfrac{1}{|\la_{n+im_0-1}|}>\dfrac{2c_0}{c_2}$
\item[(2.5)] $n\bigg(\bigg|\dfrac{\la_{n+1}}{\la_n}\bigg|-1\bigg)>2c_3$
\item[(2.6)] $\dfrac{n}{n+m_0k_0}>\dfrac{1}{2}$
\item[(2.7)] $\dfrac{n}{|\la_n|}\cdot2c_0<c_4$
\item[(2.8)] $\dfrac{n}{|\la_n|}<\dfrac{c_4}{2c_2k_0}$.
\end{enumerate}
Of course, inequality (2.6) has nothing to do with the sequence $(\la_n)$; however, we chose to isolate it here
since it will be needed later in the main construction of the partition and in the construction of the disks. At
a first glance, it may look strange why the above properties play important role. It turns out that these
properties fully characterize the sequences $( \la_n)$ that appear in Theorem \ref{thm1}, see Lemma
\ref{sec7lem3} in section 7.

\section{Construction of the partition of the sector $S$}\label{sec3}
For the rest of this section we fix a sequence $(\la_n)$ of non-zero complex numbers satisfying the following:
\begin{enumerate}
\item[1)] $|\la_{n+1}|-|\la_n|\ra+\infty$ as $n\ra+\infty$
\item[2)] $\dfrac{\la_{n+1}}{\la_n}\ra1$ as $n\ra+\infty$
\item[3)] $\underset{n\ra+\infty}{\dis\lim\inf}\Big(n\Big(\Big|\dfrac{\la_{n+1}}{\la_n}\Big|
    -1\Big)\Big)>0$.
\end{enumerate}
We also fix the numbers $r_0,R_0,\thi_0,\thi_T$,$c_0$, $c_1$, $c_2$, $c_3$, $c_4$, $m_0$, $k_0$, which are
defined in subsection \ref{good}.

\subsection{Step 1. Partitions of the interval $[\theta_0 ,\theta_T]$}
In this step we succeed the elementary structure of our construction. All the following steps are based in this
first one. For every positive integer $m$ we shall construct a corresponding partition $\Delta_m$. So, let $m\in
\mathbb{N}$ be fixed. We have (see subsection \ref{good})
$$m_0:=\left[\frac{R_0c_1}{r_0}\right]+1.$$
Recall that the symbol $[x]$ stands for the integer part of the real number $x$. For every $j=0,1,\ldots ,m_0-1$
choose real numbers $\theta_j^{(m)}$, $\theta_{j+1}^{(m)}$ such that
$$\theta_j^{(m)}, \,\,\theta_{j+1}^{(m)}\in [\theta_0,\theta_T)$$
and
$$
\frac{c_0}{2R_0c_1|\lambda_{m+j}|}<\theta_{j+1}^{(m)}-\theta_j^{(m)}<\frac{c_0}{R_0c_1|\lambda_{m+j}|} \quad
\quad (I)
$$
where $\theta_0^{(m)}=\theta_0$. We consider three cases.

{\bf Case 1.} Assume that $$\frac{c_0}{2R_0c_1|\lambda_{m}|}\geq \theta_T-\theta_0 .$$ Then we define
$$\Delta_m=\{ \theta_0^{(m)} \}.$$

{\bf Case 2.} Assume that
$$
\frac{c_0}{2R_0c_1}\sum_{j=0}^{j'}\frac{1}{|\lambda_{m+j}|}\geq \theta_T-\theta_0
$$
for some $j'\in \{ 1, \ldots , m_0 \}$. Consider the lowest number $j_0 \in \{ 1,\ldots ,m_0\}$ so that the
previous inequality holds. Then we define our partition to be
$$\Delta_m=\{ \theta_j^{(m)}: j=0,\ldots ,j_0-1 \} .$$

{\bf Case 3.} Assume that none of the Cases 1, 2 hold. Then by inequality (I) we get that
$$\theta_0=\theta_0^{(m)}<\theta_1^{(m)}<\cdots <\theta_{m_0}^{(m)}<\theta_T .$$
Setting $\sigma_m:=\theta_{m_0}^{(m)}-\theta_0$ we have $0<\sigma_m<\theta_T-\theta_0$. For every positive
integer $k$ with $k\geq m_0+1$ there exist unique $\nu \in \mathbb{N}$ and unique $j\in \{ 0, 1,\ldots ,m_0-1\}$
such that $k=\nu m_0+j$. For every $k$ as before, set
$$\theta_k^{(m)}=\theta_{\nu m_0+j}^{(m)}:=\theta_j^{(m)}+\nu \sigma_m .$$
It is obvious that the sequence $(\theta_k^{(m)})_{k=1}^{\infty}$ is strictly increasing and tends to $+\infty$.
Without loss of generality we may assume that
$$
\theta_k^{(m)}\neq \theta_T \,\,\,\textrm{for every}\,\,\, k\geq m_0+1.
$$
Otherwise, if $\theta_{k'}^{(m)}=\theta_T$ for some $k'\geq m_0+1$, and since $(\theta_k^{(m)})_{k=1}^{\infty}$
is strictly increasing, $k'$ is the only integer having this property. Then we subtract a sufficiently small
positive number $\epsilon >0$ from $\theta_{k'}^{(m)}$ so that replacing $\theta_{k'}^{(m)}$ by
$\theta_{k'}^{(m)}-\epsilon$ in the sequence $(\theta_k^{(m)})_{k=1}^{\infty}$, inequality (I) still holds.

Finally we define $\nu_m$ to be the biggest integer $\nu $ with the properties $\nu\geq m_0+1$ and
$\theta_{\nu}^{(m)}<\theta_T$. We are ready to describe the desired partition $\Delta_m$.
$$\Delta_m:=\{ \theta_0^{(m)}, \theta_1^{(m)}, \ldots ,\theta_{ \nu_{m}}^{(m)} \} .$$

The partitions, $\Delta_1, \Delta_2,\ldots $ constructed above can be chosen so that the following important
property holds:

$$ \textrm{ "almost disjoint property" if} \,\,\, m_1\neq m_2 \,\,\, \textrm{then} \,\,\,
\Delta_{m_1}\cap \Delta_{m_2}=\{ \theta_0 \} .$$

The "almost disjoint property" turns out to be very important in the rest of the construction.

\subsection{Step 2. Partitions of the arc $\phi_r([\theta_0 ,\theta_T])$}
Consider the function $\phi :[\theta_0 ,\theta_T ]\times (0 ,+\infty )\to \mathbb{C}$ given by
$$\phi (t,r):=re^{2\pi it},\,\,\, (t,r)\in [\theta_0 ,\theta_T ]\times (0 ,+\infty ) $$ and for every $r> 0$ we
define the corresponding curve $\phi_r:[\theta_0 ,\theta_T ] \to \mathbb{C}$ by
$$\phi_r(t):=\phi (t,r), \,\,\, t\in [\theta_0 ,\theta_T ].$$
For any given positive integer $m$, $\phi_r(\Delta_m)$ is a partition of the arc $\phi_r([\theta_0 ,\theta_T])$,
where $\Delta_m$ is the partition of the interval $[\theta_0 ,\theta_T]$ constructed in Step 1. For every $r>0$,
$m\in \mathbb{N}$ define
$$P_0^{r,m}:=\phi_r(\Delta_m)$$
which we call partition of the arc $\phi_r([\theta_0 ,\theta_T])$ with height $r$, density $m$ and order $0$.

\subsection{Step 3. Partitions of order 1 for a sector of opening $\pi /2$ }
In this step we elaborate on the construction of Step 2 and we aim to define a suitable partition for a sector
of opening $\pi /2$. For reasons that will become apparent later on, this partition is called a partition of
order 1. To explain why we deal with such a sector, notice that $\theta_T-\theta_0=1/4$. Therefore the set $\phi
([\theta_0 ,\theta_T] \times (0,+\infty ))$ is nothing else but a sector of opening $\pi /2$, where $\phi$ is
defined in Section 2.

We continue with the construction of the desired partition. Recall that
$$k_0:=\left[ \frac{2c_0}{c_2} \right] +1 .$$ The fixed positive constant $c_2$ appears in
Section 2. For every $r>0$, $m\in \mathbb{N}$ and $k\in \{ 0, 1,\ldots , k_0-1\}$ define the positive numbers
$$\mu (r,m,k):=r+\sum_{j=1}^k\frac{c_2}{|\lambda_{m+jm_0-1}|}, \,\,\, k\geq 1$$
$$\mu (r,m,0):=r.$$
Roughly speaking, our new partition will be obtained as a suitable finite union of partitions of order $0$ with
different heights and densities. More precisely for every $m\in \mathbb{N}$, $r>0$ define the set
$$P_1^{r,m}:=\bigcup_{k=0}^{k_0-1}P_0^{\mu (r,m,k), m+km_0} ,$$
where $P_0^{\mu (r,m,k), m+km_0}$ is the partition of the arc $\phi_{\mu (r,m,k)}([\theta_0 ,\theta_T])$ with
height $\mu (r,m,k)$, density $ m+km_0$ and order $0$. We call the set $P_1^{r,m}$ a partition with basis $r$,
density $m$ and order $1$. Observe that in this way we obtain the first partition in two dimensions, that is a
partition of a sector. We will built our next two dimensional partition by stacking several partitions of order
$1$.
\subsection{Step 4. Stacking several partitions of order 1: Partitions of order 2}
The positive number
$$c_4:=\frac{r_0c_3}{2}$$ is fixed from the beginning of this section.
For every positive integer $m$ and every $r>0$ we define the positive number
$$\mu_1(m):=\sum_{j=1}^{k_0}\frac{c_2}{|\lambda_{m+jm_0-1}|}$$ and observe that by Step $3$ we have
$$\mu_1(m)=\mu (r,m,k_0-1)+\frac{c_2}{|\lambda_{m+k_0m_0-1}|}-r, $$ for every $r>0$.
Let $r>0$ and $m\in \mathbb{N}$. We shall describe the new partition corresponding to $r,m$.

{\bf Case 1.} Assume that $$ \mu_1(m)\geq \frac{c_4}{m}.$$ Then we stop the process and define
$$P_{2}^{r,m}:=P_{1}^{r,m},$$
where $P_{1}^{r,m}$ is the partition defined in Step $3$.

{\bf Case 2.} Assume that $$ \mu_1(m)< \frac{c_4}{m}.$$ It trivially follows that $|w|<r+\frac{c_4}{m}$ for
every $w\in P_1^{r,m}$. Consider now the following partitions of order $1$
$$P_1^{r+\nu \mu_1(m),m}, \,\,\, \textrm{for every} \,\,\, \nu =0,1,2,\ldots .$$
Then for every $w\in P_1^{r+\nu \mu_1(m),m}$, $\nu =0,1,2,\ldots $ we get
\begin{equation} \label{St4ineq3}
|w|\geq r+\nu\mu_1(m).
\end{equation}
Let us consider the following subset of the positive integers
$$\mathcal{A}:= \left\{ N\in \mathbb{N}|\,\, |w|<r+\frac{c_4}{m}, \,\,\,\, \forall w\in \bigcup_{\nu =0}^{N}P_1^{r+\nu
\mu_1(m),m} \right\} .$$ Since $r+\nu \mu_1(m)\to +\infty$ as $\nu\to +\infty$, $|w|<r+\mu_1(m)$ for every $w\in
P_1^{r,m}$ and in view of (\ref{St4ineq3}) we conclude that the set $\mathcal{A}$ is non-empty and finite. Take
the biggest integer in this set, i.e.,
$$ \nu_0^{r,m}:=\max\mathcal{A} .$$ This integer describes the stopping time of the process. Then define the set
$$P_2^{r,m}:= \bigcup_{\nu =0}^{\nu_0^{r,m}}P_1^{r+\nu\mu_1(m),m} .$$

Throughout the rest of the paper we call the set $P_2^{r,m}$ a partition with basis $r$, density $m$ and order
$2$.
\subsection{Step 5. The final partition}
In this step we complete the construction of the desired partition of $S$. For every positive integer $m$ with
$m\geq n_0$ and every $r>0$ define the positive numbers
$$M^{r,m}:=\max \{ |w|: w\in P_2^{r,m} \}, $$ where $P_2^{r,m}$ is the partition with basis $r$, density $m$ and order
$2$ and is defined in Step $4$. We call the number $M^{r,m}$ height of the partition $P_2^{r,m}$ and define the
number
$$ l(P_2^{r,m}):= M^{r,m}-r,$$ which we call length of the partition $P_2^{r,m}$. We postpone the proof that length of
the partition $P_2^{r,m}$ is positive till the next subsection. Let us now consider the sequence
$(r_{\nu}^{(m)})$ of positive numbers, defined recursively as follows:
$$ r_0^{(m)}:=r_0,$$
$$r_1^{(m)}-r_0^{(m)}:=l(P_2^{r,m}),$$
$$r_{\nu +1}^{(m)}-r_{\nu }^{(m)}:=l(P_2^{ r_{\nu }^{(m)} ,m+\nu k_0m_0})$$
for every $\nu =1,2,\ldots $. We will prove in the next subsection that $r_{\nu }^{(m)}\to +\infty$ as $\nu\to
+\infty$ for every $m\geq n_0$. Therefore for every $m\geq n_0$ there exists a positive integer $\nu (m)$ such
that $r_{\nu (m)}^{(m)}\geq R_0$. Let $\nu_1^{(m)}$ be the smallest positive integer with the previous property.
Now we define
$$P_m:= S\bigcap \left(\bigcup_{\nu =0}^{\nu_1^{(m)}}P_2^{ r_{\nu }^{(m)} ,m+\nu k_0m_0} \right) ,$$
for every positive integer $m$ with $m\geq n_0$. For every $m$ as before the set $P_m$ defines a partition of
the sector $S$ and throughout the rest of this work $P_m$ will be called the partition of $S$ with order $m$.

\subsection{Properties of the partitions}

In the next lemma we transfer the "almost disjoint property" of the partitions of interval $[\theta_0,
\theta_T)$ to an arc.
\begin{lem} \label{propertiesl0}
Consider the partitions $P_0^{r,m_1}$, $P_0^{r,m_2}$ for given $r>0$ and $m_1,m_2\in \mathbb{N}$. The following
property holds:
$$ \textrm{ "almost disjoint property" if} \,\,\, m_1\neq m_2 \,\,\, \textrm{then} \,\,\,
P_0^{r,m_1}\cap P_0^{r,m_2}=\{ re^{2\pi i\theta_0} \} .$$
\end{lem}
\begin{proof}
The result immediately follows by the corresponding property of the partitions of the interval $[\theta_0,
\theta_T)$ and the definition of the partition $P_0^{r,m}$, see Steps 1,2.
\end{proof}

\begin{lem} \label{propertiesl1}
Consider the partition $P_1^{r,m}:=\cup_{k=0}^{k_0-1}P_0^{\mu (r,m,k), m+km_0}$ defined in Step 3, for fixed
$r>0$ and $m\in \mathbb{N}$. Take $k_1,k_2\in \{ 0, \ldots , k_0-1\}$ with $k_1< k_2$.  Then we have
$$
\mu (r,m,k_1)< \mu (r,m,k_2),
$$
where $\mu (r,m,k_1), \mu (r,m,k_2)$ are the heights of the partitions $P_0^{\mu (r,m,k_1), m+k_1m_0}$,
$P_0^{\mu (r,m,k_2), m+k_2m_0}$ respectively. In particular
$$
P_0^{\mu (r,m,k_1), m+k_1m_0}\cap P_0^{\mu (r,m,k_2), m+k_2m_0}=\emptyset .
$$
\end{lem}
\begin{proof}
By the definition of $\mu (r,m,k)$, see Step 3, it follows that $\mu (r,m,k_1)< \mu (r,m,k_2)$.
\end{proof}

\begin{lem} \label{propertiesl2}
Consider the partition $P_2^{r,m}:= \cup_{\nu =0}^{\nu_0^{r,m}}P_1^{r+\nu\mu_1(m),m}$ defined in Step 4, for
fixed $r>0$ and $m\in \mathbb{N}$. Take $\nu_1, \nu_2\in \{ 0,\ldots ,\nu_0^{r,m} \}$ with $\nu_1< \nu_2$. Then
we have
$$
\max\left\{ |w|: w\in P_1^{r+\nu_1\mu_1(m),m} \right\} <\min \left\{ |w|: w\in P_1^{r+\nu_2\mu_1(m),m} \right\}
.
$$
In particular
$$P_1^{r+\nu_1\mu_1(m),m}\cap P_1^{r+\nu_2\mu_1(m),m}=\emptyset .$$
\end{lem}
\begin{proof}
Take any $k_1,k_2\in \{ 0, \ldots ,k_0-1\}$. We have
$$
\mu(r+\nu_1\mu_1(m),m,k_1)=r+\nu_1\mu_1(m)+\sum_{j=1}^{k_1}\frac{c_2}{|\lambda_{m+jm_0-1}|}<r+(\nu_1+1)\mu_1(m)
$$
$$
\leq r+\nu_2\mu_1(m)\leq r+\nu_2\mu_1(m)+\sum_{j=1}^{k_2}\frac{c_2}{|\lambda_{m+jm_0-1}|}=\mu
(r+\nu_2\mu_1(m),m,k_2),
$$
where in the case $k_1=0$ or $k_2=0$ the corresponding sum above disappears. The last implies that the height of
any partition of order $0$ used to build the partition $P_1^{r+\nu_1\mu_1(m),m}$ is strictly lower from the
height of every partition of order $0$ used to build the partition $P_1^{r+\nu_2\mu_1(m),m}$. The conclusion
follows.

\end{proof}

\begin{lem} \label{strincr}
Fix any positive integer $m$ with $m\geq n_0$. Then for every positive number $r$ the length of the partition
$P_2^{r,m}$, i.e. the number $l(P_2^{r,m})$ defined in Step 5, is positive and hence the sequence
$(r_{\nu}^{(m)})_{\nu =0}^{\infty}$, defined in Step 5, is strictly increasing.
\end{lem}
\begin{proof}
Recall that $r_0^{(m)}:=r_0>0$ and $r_{\nu +1}^{(m)}-r_{\nu }^{(m)}:=l(P_2^{ r_{\nu }^{(m)} ,m+\nu k_0m_0})$,
see Step 5. Hence, it suffices to show that $l(P_2^{ r_{\nu }^{(m)} ,m+\nu k_0m_0})>0$. On the other hand by the
definition of the length of the partition $P_2^{r,m}$ we have
$$l(P_2^{r,m}):=M^{r,m}-r,$$
where
$$M^{r,m}:=\max \left\{ |w|: w\in P_2^{r,m} \right\} .$$
The partition $P_2^{r,m}$ is defined as a union of partitions $P_1^{r',m'}$ for certain $r',m'$. Pick such a
$P_1^{r',m'}$ which in turn is defined as a union of partitions $P_0^{r'',m''}$ for certain $r'',m''$. By the
choice of $k_0$ we conclude that $P_1^{r',m'}$ contains at least five partitions $P_0^{r'',m''}$ with pairwise
different heights, hence by Lemma \ref{propertiesl1} we get
$$ \min\left\{ |w|: w\in P_1^{r',m'} \right\} < \max\left\{ |w|: w\in P_1^{r',m'}\right\} .$$
Observe now that $$r\leq  \min\left\{ |w|: w\in P_1^{r',m'} \right\} \,\,\, \textrm{and} \,\,\, \max\left\{ |w|:
w\in P_1^{r',m'}  \right\} \leq M^{r,m}.$$ The above inequalities imply that $l(P_2^{r,m})>0$ and this completes
the proof of the lemma.
\end{proof}

\begin{lem} \label{propertiesl3}
Consider the partition $P_m:= S\bigcap \left(\bigcup_{\nu =0}^{\nu_1^{(m)}}P_2^{ r_{\nu }^{(m)} ,m+\nu k_0m_0}
\right)$ defined in Step 5, for fixed $m\in \mathbb{N}$ with $m\geq n_0$. Take $\nu_1, \nu_2\in \{ 0,\ldots
,\nu_1^{(m)} \}$, with $\nu_1<\nu_2$ and $\nu_2-\nu_1 \geq 2$. Then we have
$$
\max\left\{ |w|: w\in P_2^{ r_{\nu_1 }^{(m)} ,m+\nu_1 k_0m_0} \right\} < \min \left\{ |w|: w\in P_2^{ r_{\nu_2
}^{(m)} ,m+\nu_2 k_0m_0} \right\} .
$$
In particular,
$$
P_2^{ r_{\nu_1 }^{(m)} ,m+\nu_1 k_0m_0} \cap P_2^{ r_{\nu_2}^{(m)} ,m+\nu_2 k_0m_0} =\emptyset.
$$

\end{lem}
\begin{proof}
We proceed by induction on $\nu\in \{ 0,\ldots ,\nu_1^{(m)} \}$. Recall the following quantities by Step 5:

\begin{equation} \label{disj1}
M^{r,m}:=\max \{ |w|: w\in P_2^{r,m} \} ,
\end{equation}

\begin{equation} \label{disj2}
 l(P_2^{r,m}):= M^{r,m}-r ,
\end{equation}

\begin{equation} \label{disj3}
r_0^{(m)}:=r_0,\,\, r_{\nu +1}^{(m)}-r_{\nu }^{(m)}:=l(P_2^{ r_{\nu }^{(m)} ,m+\nu k_0m_0}).
\end{equation}
Applying (\ref{disj1}), (\ref{disj2}), (\ref{disj3}) we get
\begin{equation} \label{disj4}
M^{r_{\nu}^{(m)},m+\nu k_0m_0}=\max \left\{ |w|: w\in P_2^{ r_{\nu }^{(m)} ,m+\nu k_0m_0} \right\}
\end{equation}
$$
=r_{\nu}^{(m)}+l(P_2^{ r_{\nu }^{(m)} ,m+\nu k_0m_0})=r_{\nu +1}^{(m)}
$$
for $\nu\in \{ 0,\ldots ,\nu_1^{(m)} \}$. Using (\ref{disj4}) and the fact that the sequence
$(r_{\nu}^{(m)})_{\nu =0}^{\infty}$ is strictly increasing, see Lemma \ref{strincr}, we have
\begin{equation} \label{disj5}
\max\left\{ |z|: z\in P_2^{ r_{\nu_1 }^{(m)} ,m+\nu_1 k_0m_0} \right\} =M^{r_{\nu_1}^{(m)},m+\nu_1
k_0m_0}=r_{\nu_1+1}^{(m)}<r_{\nu_2}^{(m)}.
\end{equation}
Combining the last with the (trivial) equality
\begin{equation} \label{disj6}
r_{\nu_2}^{(m)}= \min \left\{ |w|: w\in P_2^{ r_{\nu_2 }^{(m)} ,m+\nu_2 k_0m_0} \right\}
\end{equation}
the result follows.
\end{proof}

\begin{lem} \label{propertiesl4}
Consider the partition $P_m:= S\bigcap \left(\bigcup_{\nu =0}^{\nu_1^{(m)}}P_2^{ r_{\nu }^{(m)} ,m+\nu k_0m_0}
\right)$ defined in Step 5, for fixed $m\in \mathbb{N}$ with $m\geq n_0$ and $\nu_1^{(m)}\geq 1$. Take $\nu \in
\{ 0,\ldots ,\nu_1^{(m)}-1 \}$. Then we have
$$
\max\left\{ |w|: w\in P_2^{ r_{\nu }^{(m)} ,m+\nu k_0m_0} \right\} = \min \left\{ |w|: w\in P_2^{ r_{\nu +1
}^{(m)} ,m+(\nu+1) k_0m_0} \right\}
$$
and
$$
P_2^{ r_{\nu }^{(m)} ,m+\nu k_0m_0} \cap P_2^{ r_{\nu +1 }^{(m)} ,m+(\nu+1) k_0m_0} =\{ r_{\nu +1 }^{(m)}e^{2\pi
i\theta_0} \} .
$$
\end{lem}
\begin{proof}
By (\ref{disj5}), (\ref{disj6}) we get
$$
\max\left\{ |z|: z\in P_2^{ r_{\nu }^{(m)} ,m+\nu k_0m_0} \right\} =r_{\nu+1}^{(m)}= \min \left\{ |w|: w\in
P_2^{ r_{\nu +1 }^{(m)} ,m+(\nu +1) k_0m_0} \right\}.
$$
Observe that $P_2^{ r_{\nu }^{(m)} ,m+\nu k_0m_0}=\cup_{r',m'}P_0^{r',m'}$ and $P_2^{ r_{\nu +1 }^{(m)}
,m+(\nu+1) k_0m_0}=\cup_{r'',m''}P_0^{r'',m''}$. Therefore the partitions $P_2^{ r_{\nu }^{(m)} ,m+\nu k_0m_0}$,
$P_2^{ r_{\nu +1 }^{(m)} ,m+(\nu+1) k_0m_0}$ have non-empty intersection if and only if
$$P_0^{r',m'} \cap P_0^{r'',m''} \neq \emptyset \,\,\, \textrm{for some} \,\,\,r',r'',m',m''.$$
Clearly two partitions $P_0^{r',m'}$ $P_0^{r'',m''}$ of zero order have non-empty intersection if and only if
the heights $r',r''$ are the same. In our case the last happens if and only if $r'=r''=r_{\nu+1}^{(m)}$. On the
other hand it is not difficult to see that in every partition $P_2^{r,m}=\cup P_0^{R,M}$ there do not exist
$P_0^{R,M_1}$, $P_0^{R,M_2}$ ``members" of $P_2^{r,m}$ with $M_1\neq M_2$. Hence, by the definition of the
partition of order $2$, we have that the partition of order $0$ and height $r_{\nu+1}^{(m)}$ which is a member
of $P_2^{ r_{\nu }^{(m)} ,m+\nu k_0m_0}$ is the one with density $m+(\nu+1)k_0m_0-m_0$. In a similar manner we
have that the partition of order $0$ and height $r_{\nu+1}^{(m)}$ which is a member of $P_2^{ r_{\nu +1 }^{(m)}
,m+(\nu +1)k_0m_0}$ is the one with density $m+(\nu +1)k_0m_0$. Since $m+(\nu+1)k_0m_0-m_0<m+(\nu +1)k_0m_0$, by
Lemma \ref{propertiesl0} it follows that
$$
P_0^{r_{\nu +1 }^{(m)}, m+(\nu+1)k_0m_0-m_0} \cap P_0^{r_{\nu +1 }^{(m)}, m+(\nu +1)k_0m_0} = \{ r_{\nu +1
}^{(m)}e^{2\pi i\theta_0} \}
$$
and this finishes the proof of the lemma.

\end{proof}
\begin{lem}
Fix a positive integer $m$ with $m\geq n_0$. Then the sequence $( r_{\nu}^{(m)} )_{\nu =1}^{\infty }$, defined
in Step 4, is strictly increasing and $$ \lim_{\nu \to +\infty }r_{\nu}^{(m)}=+\infty .$$
\end{lem}
\begin{proof}
We shall prove that
\begin{equation} \label{incrinfty}
 r_{\nu +1}^{(m)}-r_{\nu}^{(m)}> \frac{c_4}{2(m+\nu k_0m_0)}, \,\,\, \textrm{for every} \,\,\, \nu =0,1,\ldots .
\end{equation}
Fix $\nu \in \{ 0,1,\ldots \}$ and in order to simplify notation set
$$ r:=r_{\nu}^{(m)}, \,\,\, r':=r_{\nu +1}^{(m)}, \,\,\, m_1:=m+\nu k_0m_0.$$
By definition, see Step 4, we have
$$r'-r:=l(P_2^{r,m_1})$$
and again by definition, see Step 4, and since $( |\lambda_n|)$ is strictly increasing we get
\begin{equation} \label{ineq1incrinfty}
\mu_1(m_1):=\sum_{j=1}^{k_0}\frac{c_2}{|\lambda_{m_1+jm_0-1}|} <\frac{k_0c_2}{|\lambda_{m_1}|}.
\end{equation}
By the definition of the partition $P_2^{r,m_1}$ we obtain the inequality
$$r+l(P_2^{r,m_1})+\mu_1(m_1)\geq r+\frac{c_4}{m_1}, $$ which, in view of (\ref{ineq1incrinfty}), gives the
following lower bound on the length of $P_2^{r,m_1}$:
\begin{equation} \label{ineq2incrinfty}
l(P_2^{r,m_1})>\frac{c_4}{m_1}-\frac{k_0c_2}{|\lambda_{m_1}|} .
\end{equation}
By (2.8) we have
$$
\frac{m_1}{|\lambda_{m_1}|}<\frac{c_4}{2c_2k_0}.
$$
Combining the last inequality with (\ref{ineq2incrinfty}) we get
$$
r'-r:=l(P_2^{r,m_1})>\frac{c_4}{2m_1},
$$
which proves (\ref{incrinfty}). Clearly (\ref{incrinfty}) implies that $ \lim_{\nu \to +\infty
}r_{\nu}^{(m)}=+\infty $.

\end{proof}

\section{Construction and properties of the disks} \label{section4}
For the rest of this section we fix a sequence $(\la_n)$ of non-zero complex numbers satisfying the following:
\begin{enumerate}
\item[1)] $|\la_{n+1}|-|\la_n|\ra+\infty$ as $n\ra+\infty$
\item[2)] $\dfrac{\la_{n+1}}{\la_n}\ra1$ as $n\ra+\infty$
\item[3)] $\underset{n\ra+\infty}{\dis\lim\inf}\Big(n\Big(\Big|\dfrac{\la_{n+1}}{\la_n}\Big|
    -1\Big)\Big)>0$.
\end{enumerate}
We also fix the numbers $r_0,R_0,\thi_0,\thi_T$,$c_0$, $c_1$, $c_2$, $c_3$, $c_4$, $m_0$, $k_0$, which are
defined in subsection \ref{good}. Finally, on the basis of the above, for every positive integer $m$ we consider
the partition $P_m$ constructed in the previous section.

\subsection{Construction of the disks}
The strategy in this subsection is to construct a certain family of pairwise disjoint disks, which will allow us
later to apply successfully Runge's theorem in order to prove Proposition $2.1$. What we are going to do is to
assign to each point $w$ of the partition $P_m$ a suitable closed disk with center $w\lambda (w)$ and radius
$c_0$ (the radius will be the same for every member of the family of the disks), where $\lambda(w)$ will be
chosen from the sequence $(\lambda_n)$. We shall see that, the construction of the partition $P_m$ ensures on
the one hand that the points of the partition are close enough to each other on the sector $S$ and on the other
hand that the disks centered at these points with fixed radius $c_0$ are pairwise disjoint. This is the hard
part of our argument and also shows that the required construction is very delicate. So, let us begin with the
construction of the disks.

We set
$$B:=\{ z\in \mathbb{C}|\,\,\, |z|\leq c_0 \} .$$
Fix a positive integer $m\geq n_0$ and let $w$ be any point of the partition $P_m$ of the sector $S$. We
distinguish two cases.

{\bf Case 1.} Assume that $$w\in  \{ r_{\nu }^{(m)}e^{2\pi i\theta_0 }, \,\,\, \nu =1, \ldots ,\nu_1^{(m)} \}.$$
Then $w=r_{\nu }^{(m)}e^{2\pi i\theta_0 }$ for some $\nu \in \{ 1,\ldots ,\nu_1^{(m)} \}$. We define
$$\lambda (w):=\lambda_{m+\nu k_0m_0-m_0}$$ and
$$ B_w:=B+w\lambda (w) .$$

{\bf Case 2.} Assume that $$w\in P_m \setminus \{ r_{\nu }^{(m)}e^{2\pi i\theta_0 }, \,\,\, \nu =1, \ldots
,\nu_1^{(m)} \}.$$ By Lemmas \ref{propertiesl3}, \ref{propertiesl4}, there exists a unique $\nu \in \{
0,1,\ldots , \nu_1^{(m)} \}$ such that
$$ w\in P_2^{r_{\nu }^{(m)},m+\nu k_0m_0}=\cup_{r',m'}P_0^{r',m'}.$$
Applying Lemmas \ref{propertiesl1}, \ref{propertiesl2}, \ref{propertiesl3}, \ref{propertiesl4}, we conclude that
there is a unique pair $(r',m')$ such that
$$ w=r'e^{2\pi i \theta_k^{(m')}} \,\,\, \textrm{for some} \,\,\, k\in \{ 0, 1, \ldots ,\nu_{m'} \} .$$
Observe that $k$ can be uniquely written in the form
$$k=\rho m_0+j, \,\,\, \textrm{for some} \,\,\, \rho \in \mathbb{N} ,\,\,\, j\in \{ 0, \ldots , m_0-1 \}.$$
From the above and the definition of the partition $\Delta_{m'}$, see Step 1, we have
$$ \theta_k^{(m')}=\theta_{\rho m_0+j}^{(m')}=\theta_j^{(m')}+\rho \sigma_{m'}.$$
Finally we define
$$\lambda (w):=\lambda_{m'+j}$$ and
$$ B_w:=B+w\lambda (w) .$$
Therefore for every $w\in P_m$ we assigned a disk $B_w$ according to the above rules. This completes the desired
construction of the disks assigned to the partition $P_m$.

\subsection{Properties of the disks}
Our aim in this subsection is to prove that for a fixed positive integer $m$, the disks $B_w$ for $w\in P_m$
(corresponding to the partition $P_m$), that have been constructed in the previous subsection, are pairwise
disjoint.

\begin{lem} \label{diskprop1}
Fix a positive integer $m$ with $m\geq n_0$. Then we have
$$ B\cap B_w=\emptyset \,\,\, \textrm{for every} \,\,\, w\in P_m .$$
\end{lem}
\begin{proof}
Take $w\in P_m$. The closed disks $B$, $B_w$ are centered at $0, w\lambda (w)$ respectively and they have the
same radius $c_0$. Hence, we have to show that $|w\lambda (w)|>2c_0$. Since $|w|\geq r_0$ it suffices to prove
that $$|\lambda (w)|>\frac{2c_0}{r_0} .$$ Observe now that, by the definition of $\lambda (w)$ in the previous
subsection, $\lambda (w)=\lambda_n$ for some positive integer $n$ with $n\geq m\geq n_0$. By property (2.3) we
conclude that $|\lambda_n|>2c_0/r_0$ and this finishes the proof of the lemma.
\end{proof}

\begin{lem} \label{diskprop2}
Fix a positive integer $m$ with $m\geq n_0$.
$$ \textrm{If} \,\,\, w_1,w_2\in P_m \,\,\, \textrm{with} \,\,\, w_1\neq w_2, \,\,\, |w_1|\leq |w_2|, \,\,\, |\lambda
(w_1)| <|\lambda (w_2)| $$
$$\textrm{then} \,\,\, B_{w_1}\cap B_{w_2}=\emptyset .$$
\end{lem}
\begin{proof}
Take $w_1,w_2$ satisfying the hypothesis of the lemma. We need to show that $|w_1\lambda (w_1)-w_2\lambda
(w_2)|>2c_0$. Observe that $\lambda(w_j)=\lambda_{n_j}$ for some positive integer $n_j\geq m \geq n_0$, $j=1,2$.
Since $|\lambda (w_1)| <|\lambda (w_2)|$ and the sequence $(|\lambda_n|)$ is strictly increasing we conclude
that $n_1<n_2$. We have
$$ |w_1\lambda (w_1)-w_2\lambda (w_2)|\geq  ||w_1\lambda (w_1)|-|w_2\lambda (w_2)||=|w_2\lambda (w_2)|-|w_1\lambda
(w_1)|$$ $$\geq r_0 (|\lambda (w_2)|-|\lambda (w_1)|)=r_0 (|\lambda_{n_2}|-|\lambda_{n_1}|)\geq r_0
(|\lambda_{n_1+1}|-|\lambda_{n_1}|)>2c_0, $$ where the last inequality above follows by property (2.2).
\end{proof}

\begin{lem} \label{diskprop3}
Fix a positive integer $m$ with $m\geq n_0$.
$$\textrm{If} \,\,\, w_1,w_2\in P_m \,\,\, \textrm{with} \,\,\, w_1\neq w_2 \,\,\, \textrm{and}\,\,\, |w_1|=|w_2| \,\,\,
\textrm{then} \,\,\, B_{w_1}\cap B_{w_2}=\emptyset .$$
\end{lem}
\begin{proof}
Fix $w_1$, $w_2$ satisfying the hypothesis of the lemma. Then we have $w_1=re^{2\pi i\theta_1}$, $w_2=re^{2\pi
i\theta_2}$, for some $r\in [r_0,R_0]$ and some $\theta_1, \theta_2\in [\theta_0, \theta_T)$. Since $w_1,w_2\in
P_m=\cup_{(r',m')\in J} P_0^{r',m'}$, where $J$ is a suitable set of indices, then
$$ \textrm{either}\,\,\, w_1\in P_0^{r,m_1} \,\, \textrm{and}\,\, w_2\in P_0^{r,m_2} \,\,\,
\textrm{for} \,\,\, (r,m_1),(r,m_2)\in J, \,\,\,  m_1\neq m_2$$
$$ \textrm{or} \,\,\, w_1,w_2\in P_0^{r,m'} \,\,\, \textrm{for some} \,\,\, (r,m')\in J.$$

Let us first consider the case where $w_1$, $w_2$ belong to different partitions of zero order. Then necessarily
we have $|\lambda (w_1)| \neq |\lambda (w_2)|$ and since $|w_1|=|w_2|$, Lemma \ref{diskprop2} implies that the
disks $B_{w_1}$, $B_{w_2}$ are disjoint.

We turn now to the case where both $w_1$, $w_2$ belong to the same partition of zero order $P_0^{r,m'}$. By the
definition of the partition $P_0^{r,m'}$ there exist $k_1,k_2\in \{ 0,\ldots , \nu_{m'} \}$ such that
$$ \theta_1=\theta_{k_1}^{(m')} \,\,\, \textrm{and} \,\,\, \theta_2=\theta_{k_2}^{(m')} .$$ We also have that
$$ k_1=\rho_1m_0+j_1, \,\,\, k_2=\rho_2m_0+j_2$$
for $\rho_1,\rho_2\in \mathbb{N}$ and $j_1,j_2\in \{ 0, \ldots ,m_0-1\} $ and by the definition of the partition
$\Delta_{m'}$, see Step 1, it follows that
$$\theta_{k_1}^{(m')}=\theta_{j_1}^{(m')}+\rho_1\sigma_{m'}, \,\,\,
\theta_{k_2}^{(m')}=\theta_{j_2}^{(m')}+\rho_2\sigma_{m'}, $$ where (recall from Step 1),
$$\sigma_{m'}:=\theta_{m_0}^{(m')}-\theta_0 .$$
We shall consider two cases.

{\bf Case 1.} Assume that $j_1\neq j_2$. Since $\lambda (w_1)=\lambda_{m'+j_1}$, $\lambda
(w_2)=\lambda_{m'+j_2}$ it readily follows that $|\lambda (w_1)|\neq |\lambda (w_2)|$. In view of Lemma
\ref{diskprop2} we conclude that the disks $B_{w_1}$, $B_{w_2}$ are disjoint.

{\bf Case 2.} It remains to handle the case $j_1=j_2$. Observe that in this situation we have
\begin{equation} \label{diskprop3eq1}
\theta_2-\theta_1=(\rho_2-\rho_1)\sigma_{m'}.
\end{equation}
Since $w_1\neq w_2$ and $|w_1|=|w_2|$ we may assume with no loss of generality that $\theta_1<\theta_2$. We
establish below a "sufficiently large" lower bound on $\sigma_{m'}$. Inequality (I) in Step 1 implies the
following
\[
\thi^{(m')}_1-\thi^{(m')}_0>\frac{c_0}{2R_0c_1}\frac{1}{|\la_{m'}|}
\]
\[
\thi^{(m')}_2-\thi^{(m')}_1>\frac{c_0}{2R_0c_1}\cdot\frac{1}{|\la_{m'+1}|}
\]
\hspace*{7cm} $\vdots$ \hspace*{1.5cm} $\vdots$
\[
\thi^{(m')}_{m_0}-\thi^{(m')}_{m_0-1}>\frac{c_0}{2R_0c_1}\cdot\frac{1} {|\la_{m'+m_0-1}|} .
\]
Adding by pairs the previous inequalities we get
\begin{equation} \label{diskprop3eq2}
\sigma_{m'}:= \thi^{(m')}_{m_0}-\thi_0>\frac{c_0}{2R_0c_1}\cdot\sum^{m_0-1}_{j=0}\frac{1} {|\la_{m'+j}|}.
\end{equation}
We also need the following inequality, so called Jordan's inequality:
\begin{equation} \label{diskprop3eq3}
\sin x>\frac{2}{\pi }x, \,\,\, x\in (0,\frac{\pi }{2} ).
\end{equation}
Since $r\geq r_0$, $(|\lambda_n|)_n$ is strictly increasing, $\theta_1<\theta_2$ and by (\ref{diskprop3eq1}),
(\ref{diskprop3eq2}), (\ref{diskprop3eq3}) and property (2.1) we get
$$ |w_2\lambda (w_2)-w_1\lambda (w_1)|= r|\lambda_{m'+j_0}||e^{2\pi i\theta_2}-e^{2\pi i\theta_1}|\geq
r_0|\lambda_{m'}||e^{2\pi i\theta_2}-e^{2\pi i\theta_1}|$$
$$=r_0|\lambda_{m'}|2\sin(\pi (\theta_2-\theta_1))>2r_0|\lambda_{m'}|\frac{2}{\pi}(\pi (\theta_2-\theta_1))
=4r_0|\lambda_{m'}|(\rho_2-\rho_1)\sigma_{m'}$$
$$\geq 4r_0|\lambda_{m'}| \sigma_{m'} >
 \frac{2r_0c_0}{R_0c_1}|\lambda_{m'}|\sum^{m_0-1}_{j=0}\frac{1}{|\la_{m'+j}|} >2c_0,$$ where the last inequality follows by
property (2.1). This finishes the proof for the Case 2 and hence that of the lemma.
\end{proof}

\begin{lem} \label{diskprop4}
Fix a positive integer $m$ with $m\geq n_0$.
$$\textrm{If} \,\,\, w_1,w_2\in P_m \,\,\, \textrm{with} \,\,\, w_1\neq w_2 \,\,\, \textrm{and} \,\,\,
\lambda (w_1)=\lambda (w_2) \,\,\, \textrm{then} \,\,\, B_{w_1}\cap B_{w_2}=\emptyset .$$
\end{lem}
\begin{proof}
Fix $w_1$, $w_2$ satisfying the hypothesis of the lemma. If $|w_1|=|w_2|$ then by Lemma \ref{diskprop3} the
conclusion follows. So assume that $|w_1|\neq |w_2|$. By the definition of the partition $P_m$ there exist
$\nu_1, \nu_2\in \{ 0,\ldots , \nu_1^{(m)} \}$ such that
$$ w_1\in P_2^{r_{\nu_1}^{(m)},m+\nu_1k_0m_0},\,\,\,w_2\in P_2^{r_{\nu_2}^{(m)},m+\nu_2k_0m_0}.$$
{\bf Claim 1.} $\nu_1 =\nu_2$.

\emph{Proof of Claim 1}: We argue by contradiction, so assume that $\nu_1\neq \nu_2$. Without loss of generality
suppose that $\nu_1<\nu_2$. By the definition of the partition of order $2$ we have
$$
P_2^{r_{\nu_1}^{(m)},m+\nu_1k_0m_0}=\bigcup_{\nu =0}^{\nu_0^{ r_{\nu_1}^{(m)},m+\nu_1k_0m_0 }}
P_1^{r_{\nu_1}^{(m)}+\nu \mu_1(m+\nu_1k_0m_0),m+\nu_1k_0m_0}
$$
and
$$
P_2^{r_{\nu_2}^{(m)},m+\nu_2k_0m_0}=\bigcup_{\nu =0}^{\nu_0^{ r_{\nu_2}^{(m)},m+\nu_2k_0m_0 }}
P_1^{r_{\nu_2}^{(m)}+\nu \mu_1(m+\nu_2k_0m_0),m+\nu_2k_0m_0}.
$$
Hence there exist $\nu'\in \left\{ 0, \ldots , \nu_0^{ r_{\nu_1}^{(m)},m+\nu_1k_0m_0 } \right\}$, $\nu'' \in
\left\{ 0,\ldots ,\nu_0^{ r_{\nu_2}^{(m)},m+\nu_2k_0m_0 } \right\}$ such that
\begin{equation} \label{claim1eq1}
w_1\in P_1^{r_{\nu_1}^{(m)}+\nu' \mu_1(m+\nu_1k_0m_0),m+\nu_1k_0m_0}
\end{equation}
and
\begin{equation} \label{claim1eq2}
w_2\in P_1^{r_{\nu_2}^{(m)}+\nu'' \mu_1(m+\nu_2k_0m_0),m+\nu_2k_0m_0}.
\end{equation}
Recall that, see Step 1, for every $r>0$ and every positive integer $m$ the partition $P_1^{r,m}$ is defined as
a union of partitions of order $0$ as follows:
$$P_1^{r,m}=\bigcup_{k=0}^{k_0-1}P_0^{\mu (r,m,k), m+km_0}.$$
Thus, by (\ref{claim1eq1}), (\ref{claim1eq2}), there exist $k_1,k_2\in \{ 0,\ldots ,k_0-1\}$ and $r_1,r_2$
positive numbers such that
$$w_1\in P_0^{r_1,m+\nu_1k_0m_0+k_1m_0},\,\,\, w_2\in P_0^{r_2, m+\nu_2k_0m_0+k_2m_0}.$$
From the last and the definition of $\lambda (w)$ for $w\in P_m$ we have
$$ \lambda (w_1)=\lambda_{n'+j'},\,\,\, \lambda (w_2)=\lambda_{n''+j''}, $$
where $n'=m+\nu_1k_0m_0+k_1m_0$, $n''=m+\nu_2k_0m_0+k_2m_0$ and $j',j''\in \{ 0,\ldots , m_0-1 \}$. Observe now
that
\begin{equation} \label{claim1eq3}
n'+l' <n''+l'' \,\,\, \textrm{for every} \,\,\, l',l''\in  \{ 0,\ldots , m_0-1 \} .
\end{equation}
By (\ref{claim1eq3}) and the fact that $(|\lambda_n |)$ is strictly increasing we arrive at
$$ |\lambda (w_1)|=|\lambda_{n'+j'}|<|\lambda_{n''+j''}|=|\lambda (w_2)| ,$$ which is a contradiction. This finishes
the proof of the Claim 1.

For simplicity reasons let us define
$$\nu:=\nu_1=\nu_2.$$
By the proof of Claim 1, we have that
$$
w_1\in P_0^{r_1,m'+k_1m_0},\,\,\, w_2\in P_0^{r_2, m'+k_2m_0}
$$
and
\begin{equation} \label{claimeq4}
\lambda_{m'+k_1m_0+j'}= \lambda (w_1)= \lambda (w_2)=\lambda_{m'+k_2m_0+j''},
\end{equation}
where $m':=m+\nu k_0m_0$, $r_1,r_2>0$ and $j',j''\in  \{ 0,\ldots , m_0-1 \}$.

{\bf Claim 2.} $k_1=k_2$.

\emph{Proof of Claim 2}: We argue by contradiction, so assume that $k_1\neq k_2$. Without loss of generality
assume that $k_1<k_2$. Then we have
$$ m'+k_1m_0+j'\leq m'+(k_1+1)m_0-1<m'+k_2m_0+j'' .$$ The last implies
that $ |\lambda_{m'+k_1m_0+j'}|< |\lambda_{m'+k_2m_0+j''}| $, which contradicts (\ref{claimeq4}).

Observe now that we also have $j'=j''$. Set $r':=r_{\nu }^{(m)}$, $j:=j'=j''$ and $k:=k_1=k_2$. Recall that
$$w_1,w_2\in P_2^{r',m'}.$$ By the proof of Claim 1 and the previous notations we immediately get the following
$$
w_1\in P_0^{\mu (r'+\nu'\mu_1(m'), m',k), m'+km_0},\,\,\,w_2\in P_0^{\mu (r'+\nu''\mu_1(m'), m',k), m'+km_0},
$$
for some $$\nu',\nu'' \in \left\{ 0, \ldots , \nu_0^{ r',m' } \right\}.$$ It is now clear that
\[
|w_1|=\mu (r'+\nu'\mu_1(m'), m',k)=r'+\nu'\mu_1(m')+\sum_{N=1}^k\frac{c_2}{|\lambda_{m'+Nm_0-1}|} ,
\] 
and

\[
|w_2|=\mu (r'+\nu''\mu_1(m'), m',k)=r'+\nu''\mu_1(m')+\sum_{N=1}^k\frac{c_2}{|\lambda_{m'+Nm_0-1}|} ,
\]
where we used the definition of $\mu (r,m,k)$ from Step $3$. It is immediate that
$$|\nu'-\nu''|\geq 1,$$ since $|w_1|\neq |w_2|$. We are ready for the final estimate. From the above we arrive at the following inequality
$$|w_1\lambda (w_1)-w_2\lambda (w_2)|\geq
|\lambda_{m'+km_0+j}|||w_1|-|w_2||=|\lambda_{m'+km_0+j}|\mu_1(m')|\nu'-\nu''|$$
$$\geq |\lambda_{m'}| \mu_1(m') =|\lambda_{m'}|\sum_{N=1}^{k_0}\frac{c_2}{|\lambda_{m'+Nm_0-1}|}>2c_0,$$
where the last inequality follows by property (2.4). This completes the proof of the lemma.

\end{proof}

\begin{lem}\label{diskprop6}
Let $m\ge n_0$, $m\in\N$, $r\in[r_0,R_0]$, $\thi'$, $\thi''\in[\thi_0,\thi_T]$ and $v_1<v_2$, where
$v_1\in\{m,m+1,\ld,m+m_0k_0-1\}$, $v_2\in \mathbb{N}$. Also, let $\e_1,\e_2$ be two non-negative real numbers
such that $0\le\e_2<\e_1<\dfrac{c_4}{m}$. We consider the numbers $r_1:=r+\e_1$ and $r_2:=r+\e_2$ and define the
discs $B(1):=B+r_1e^{2\pi i\thi'}\la_{v_1}$, $B(2):=B+r_2e^{2\pi i\thi''}\la_{v_2}$. Then $B(1)\cap
B(2)=\emptyset$.
\end{lem}
\begin{proof}
By property (2.6) we have $$\frac{m}{m+k_0m_0}>\frac{1}{2}$$ or equivalently
\begin{equation} \label{diskprop6-1}
m+k_0m_0<2m .
\end{equation}
We also have
\begin{equation} \label{diskprop6-2}
m\leq v_1 \leq m+m_0k_0-1 ,
\end{equation}
by our hypothesis. Hence, by (\ref{diskprop6-1}), (\ref{diskprop6-2}) it follows that
\begin{equation} \label{diskprop6-3}
v_1<2m.
\end{equation}
Combining (\ref{diskprop6-3}) with the definition of $c_4$ we get
\begin{equation} \label{diskprop6-4}
\frac{c_4}{m}<\frac{r_0c_3}{v_1}.
\end{equation}
By (\ref{diskprop6-4}) and our hypothesis we arrive at the following inequality
$$\e_1-\e_2<\frac{r_0c_3}{v_1}$$ or equivalently
\begin{equation} \label{diskprop6-5}
2r_0c_3-v_1(\e_1-\e_2)>r_0c_3.
\end{equation}
Since $v_1<v_2$ and $(|\lambda_n|)$ is strictly increasing we have
$$
v_1 \left( \left| \frac{\lambda_{v_2}}{\lambda_{v_1}} \right| -1\right) \geq v_1 \left( \left|
\frac{\lambda_{v_1+1}}{\lambda_{v_1}} \right| -1\right)
$$
and in view of property (2.5) (recall that $v_1\geq m\geq n_0$)
\begin{equation} \label{diskprop6-6}
rv_1 \left( \left| \frac{\lambda_{v_2}}{\lambda_{v_1}} \right| -1\right) >2rc_3\geq 2r_0c_3.
\end{equation}
By (\ref{diskprop6-5}), (\ref{diskprop6-6}) we get
\begin{equation} \label{diskprop6-7}
rv_1 \left( \left| \frac{\lambda_{v_2}}{\lambda_{v_1}} \right|
-1\right)-v_1(\e_1-\e_2)>2r_0c_3-v_1(\e_1-\e_2)>r_0c_3.
\end{equation}
Since $c_4:=r_0c_3/2$, inequality (\ref{diskprop6-7}) combined with property (2.7) gives
$$
rv_1 \left( \left| \frac{\lambda_{v_2}}{\lambda_{v_1}} \right|
-1\right)-v_1(\e_1-\e_2)>\frac{v_1}{|\lambda_{v_1}|}2c_0,
$$
or equivalently
$$
r \left( \left| \frac{\lambda_{v_2}}{\lambda_{v_1}} \right| -1\right)-(\e_1-\e_2)>\frac{1}{|\lambda_{v_1}|}2c_0.
$$
Adding on the left hand side of the previous inequality the positive term
$$
\e_2(\frac{|\lambda_{v_2}|}{|\lambda_{v_1}|}-1),
$$
we get that
$$
(r+\e_2) \left( \left| \frac{\lambda_{v_2}}{\lambda_{v_1}} \right|
-1\right)-(\e_1-\e_2)>\frac{1}{|\lambda_{v_1}|}2c_0.
$$
Multiplying both sides of the above inequality by $|\lambda_{v_1}|$ we arrive at
$$
(r+\e_2)|\lambda_{v_2}|-(r+\e_1)|\lambda_{v_1}|>2c_0,
$$
which implies that the disks $B(1)$, $B(2)$ are disjoint.
\end{proof}

\begin{lem} \label{diskprop7}
Fix a positive integer $m$ with $m\geq n_0$. Then the family $\mathcal{B}_m$, defined by
$$ \mathcal{B}_m:=\{ B_w| w\in P_m \} \cup \{ B\} ,$$
consists of pairwise disjoint disks.
\end{lem}
\begin{proof}
By Lemma \ref{diskprop1} we have that $B\cap B_w=\emptyset$ for every $w\in P_m$. So let us fix $w_1,w_2\in P_m$
with $w_1\neq w_2$. We have to show that $B_{w_1}\cap B_{w_2}=\emptyset$. If $|w_1|=|w_2|$ by Lemma
\ref{diskprop3} the conclusion follows. So, let us assume that $|w_1|\neq |w_2|$. Now we look at $\lambda
(w_1)$, $\lambda (w_2)$. If $|\lambda (w_1)|=|\lambda (w_2)|$, and keeping in mind that
$|\lambda_n|=|\lambda_{n'}|$ if and only if $\lambda_n=\lambda_{n'}$, then by Lemma \ref{diskprop4} the
corresponding disks $B_{w_1}$, $B_{w_2}$ are disjoint. It remains to deal with the case $|\lambda (w_1)|\neq
|\lambda (w_2)|$. Without loss of generality assume that $|w_1|<|w_2|$. We shall consider the following two
cases.

{\bf Case 1.} $|\lambda (w_1)|<|\lambda (w_2)|$. Then by Lemma \ref{diskprop2} we conclude that $B_{w_1}\cap
B_{w_2}=\emptyset$.

{\bf Case 2.} $|\lambda (w_1)|>|\lambda (w_2)|$. By the definition of partition $P_m$ we have that $P_m$ is a
union of partitions of order $2$, so there exist pairs $(r_1,m_1)$, $(r_2,m_2)$ for certain $r_1,r_2>0$ and
$m_1,m_2$ positive integers such that $w_1\in P_2^{r_1,m_1}$ and $w_2\in P_2^{r_2,m_2}$. If $(r_1,m_1)\neq
(r_2,m_2)$, by the proof of Claim 1 in Lemma \ref{diskprop4} and Lemma \ref{propertiesl3} it follows that
$|\lambda (w_1)|<|\lambda (w_2)|$, which is a contradiction. Therefore $w_1$, $w_2$ belong to the same partition
of order $2$, say $P_2^{r',m'}$. In order to apply Lemma \ref{diskprop6} we introduce the following "strange"
notation:
$$ r_1:= |w_2|,\,\,\, r_2:=|w_1|.$$
Since $w_1\in P_2^{r',m'}$, we have that $r_1=|w_2|=r'+\e_1$ for some $0\leq \e_1< +\infty$. By a similar
reasoning we have that $r_2=|w_1|=r'+ \e_2$ for some positive number $\e_2$. Observe that $\e_1>0$ because
$|w_1|<|w_2|$. Recall that $|w|<r'+\frac{c_4}{m'}$ for every $w\in P_2^{r',m'}$, see Step 4. On the other hand
$$M^{r',m'}:=\max \{ |w||w\in P_2^{r',m'} \} ,$$ by Step 5. Hence, we get
$$ |w_1|=r_2=r'+\e_2\leq M^{r',m'}<r'+\frac{c_4}{m'}$$
and
$$ |w_2|=r_1=r'+\e_1\leq M^{r',m'}<r'+\frac{c_4}{m'},$$
from which it follows that
\begin{equation} \label{diskprop7-1}
 \e_1< \frac{c_4}{m'}.
\end{equation}
The inequality $|w_1|<|w_2|$ implies that $\e_2<\e_1$. From the last and (\ref{diskprop7-1}) we conclude that
$$0<\e_1-\e_2\leq \e_1<\frac{c_4}{m'}.$$
We also have
$$\lambda (w_1)=\lambda_{v_2} \,\,\, ,\lambda (w_1)=\lambda_{v_1},$$
for some positive integers $v_1,v_2$ with $v_1,v_2\geq m'$, $v_1\leq m'+k_0m_0-1$ and $v_1<v_2$. Since
$w_2=r_1e^{2\pi i\theta'}$, $w_1=r_2e^{2\pi i\theta''}$ for some $\theta', \theta''\in [\theta_0 ,\theta_T)$, we
apply Lemma \ref{diskprop6} and the desired result follows. This completes the proof of the lemma.
\end{proof}

\section{Proof of Lemma \ref{lem3}}\label{sec5}
\noindent

Let some fixed $j_1,s_1,k_1\in\N$. We will prove that the set $\bigcup\limits^\infty_{m=1}E(m,j_1,s_1,k_1)$ is
dense in $(\ch(\C),\ct_u)$.

For simplicity we write $p_{j_1}=p$. Consider fixed $g\in\ch(\C)$, a compact set $C\subseteq\C$ and $\e_0>0$. We
seek $f\in\ch(\C)$ and a positive integer $m_1$ such that

\begin{equation}\label{sec5eq1}
f\in E(m_1,j_1,s_1,k_1)
\end{equation}
and
\begin{equation} \label{sec5eq2}
\sup_{z\in C}|f(z)-g(z)|<\e_0.
\end{equation}

Fix $R_1>0$ sufficiently large so that
$$
C\cup\{z\in\C|\,|z|\le k_1\}\subset\{z\in\C|\,|z|\le R_1\}.
$$
Choose $0<\de_0<1$ such that
\begin{equation} \label{sec5eq3}
\textrm{if} \,\,\, |z|\le R_1 \,\,\, \textrm{and} \,\,\, |z-w|<\de_0 \,\,\, \textrm{then} \,\,\,
|p(z)-p(w)|<\frac{1}{2s_1}.
\end{equation}

We set $$B:=\{z\in\C|\,|z|\le R_1+\de_0\},$$  $$c_0:=R_1+\de_0,$$ $$c_1:=\dfrac{4\pi(R_1+\de_0)}{\de_0} ,$$
$$c_2:=\dfrac{\de_0}{2R_0},$$ $$m_0:=\Big[\dfrac{R_0}{r_0}c_1\Big]+1,$$ $$k_0:=\Big[\dfrac{2(R_1+\de_0)}{c_2}\Big]+1,$$
$$c_3 \,\,\, \textrm{any, fixed, positive number}$$ and $$c_4:=\dfrac{r_0c_3}{2}.$$

Fix a natural number $n_0$ such that all properties (2.1)-(2.8) hold for every $n\geq n_0$ with respect to the
above fixed quantities. Let us also fix a positive integer $m\ge n_0$. After that, on the basis of the fixed
numbers $r_0,R_0,\thi_0,\thi_T,c_0,c_1,c_2,c_3,k_0,m_0$ and the natural number $m$ we define the set $L_m$ as
follows:
$$L_m:=B\cup \left( \bigcup_{w\in P_m}B_w \right) ,$$
where the discs $B_w$, $w\in P_m$ are constructed in section 4.1. By Lemma \ref{diskprop7}, the disks in the
family $\mathcal{B}_m$ are pairwise disjoint. Therefore the compact set $L_m$ has connected complement. This
property is needed in order to apply later Meregelyan's theorem. We now define the function $h$ on the compact
set $L_m$ by,
\[
h(z)=\left\{\begin{array}{cc}
              g(z), & z\in B \\
              p(z-w\lambda (w)), & z\in B_w ,w\in P_m.
            \end{array}\right.
\]
By Mergelyan's theorem \cite{12} there exists an entire function $f$ (in fact a polynomial) such that
\begin{eqnarray}
\sup_{z\in L_m}|f(z)-h(z)|<\min\bigg\{\frac{1}{2s_1},\e_0\bigg\}.  \label{sec5eq4}
\end{eqnarray}
By the definition of $h$ and (\ref{sec5eq4}) it follows that

\begin{align}
\sup_{z\in C}|f(z)-g(z)|&\le\sup_{z\in B}|f(z)-g(z)|=\sup_{z\in B}
|f(z)-h(z)|\nonumber \\
&\le\sup_{z\in L_m}|f(z)-h(z)|< \e_0  \label{sec5eq5}
\end{align}
which implies the desired inequality (\ref{sec5eq2}).

It remains to show (\ref{sec5eq1}).

Let $a\in S$. We can write $a=re^{2\pi i\theta }$ for some $r\in[r_0,R_0]$ and $t\in[\thi_0,\thi_T]$. Since
$P_m=\cup P_0^{r',m'}$, consider all the $r'$ that appear in the previous union and order them as follows:
$r_0<r_1<\cdots <r_{N}\leq R_0$ for some $N\in \mathbb{N}$. Then either there exists unique $\nu \in \{
0,1,\ldots , N-1\}$ such that $r_{\nu} \leq r<r_{\nu +1}$ or $r_N \leq r \leq R_0$. Define
$$r_1:=r_{\nu }, \,\,\, r_2:=r_{\nu +1},\,\,\, \textrm{if} \,\,\,r_{\nu} \leq r<r_{\nu +1},$$
and
$$r_1:=r_N, \,\,\, r_2:=R_0,\,\,\, \textrm{if} \,\,\,r_N \leq r \leq R_0 .$$
Observe that in either case we have $r_1\leq r\leq r_2$.

Consider now all the partitions with height $r_1$ and order $0$, $P_0^{r_1,m'}$, that appear in $P_m$. By the
construction of $P_m$ either there exists a unique $m'$ so that the partition $P_0^{r_1,m'}$ appears in $P_m$,
in other words there exists a unique partition of order $0$ with height $r_1$, or there exist exactly two
different partitions of order $0$ and height $r_1$, say $P_0^{r_1,m'}$, $P_0^{r_1,m''}$. In the latter case we
consider the partition with the biggest density, for which we use again the symbol $P_0^{r_1,m'}$.

In the above paragraph we fixed a partition of order $0$ and height $r_1$, $P_0^{r_1,m'}$. The positive integer
$m'$ reflects the density of the partition and recall that, see Step 1,
$$\Delta_{m'}:= \{ \theta_0^{(m')}, \theta_1^{(m')}, \ldots , \theta_{\nu_{m'}}^{(m')} \}.$$ It now follows that
either there exists a unique $j\in \{ 1,2,\ldots , \nu_{m'}-1\}$ such that
$$ \theta_j^{(m')}\leq \theta <\theta_{j+1}^{(m')}$$ or
$$\theta_{\nu_{m'}}^{(m')}\leq \theta \leq \theta_T.$$ Then we define
$$
\theta_1:=\theta_j^{(m')}, \,\,\, \theta_2:=\theta_{j+1}^{(m')}, \,\,\,\textrm{if} \,\,\, \theta_j^{(m')}\leq
\theta <\theta_{j+1}^{(m')}
$$
and
$$
\theta_1:=\theta_{\nu_{m'}}^{(m')}, \,\,\, \theta_2:=\theta_T, \,\,\, \textrm{if} \,\,\,
\theta_{\nu_{m'}}^{(m')}\leq \theta \leq \theta_T.
$$

Let us also define
$$w_0:=r_1e^{2\pi i\theta_1} \in P_m .$$
We will prove now that for every $z\in\C$ with $|z|\le R_1$ we have $z+a\la (w_0)\in B_{w_0}$. Recall that
$B_{w_0}:=B+w_0\lambda (w_0)$. We have $B_{w_0}=\overline{D}(w_0\lambda (w_0),R_1+\de_0)$. Thus, it suffices to
prove that
\begin{eqnarray}
|(z+a\la (w_0))-w_0\lambda (w_0)|<R_1+\de_0, \ \ \text{for} \ \ |z|\le R_1.  \label{sec5eq6}
\end{eqnarray}
For $|z|\le R_1$ we have
\begin{eqnarray}
|(z+a\la (w_0)-w_0\lambda (w_0)|\le R_1+|\la (w_0)| |re^{2\pi i\theta }-r_1e^{2\pi i\thi_1}|. \label{sec5eq7}
\end{eqnarray}
By (\ref{sec5eq7}), in order to prove (\ref{sec5eq6}) it suffices to prove that
\begin{eqnarray}
|\la (w_0)| |re^{2\pi i\thi}-r_1e^{2\pi i\thi_1}|<\de_0.  \label{sec5eq8}
\end{eqnarray}
We have now:
\begin{align*}
|re^{2\pi i\thi}-r_1e^{2\pi i\thi_1}|&\le|r_1-r_2|+R_0|e^{2\pi i\thi_1}-e^{2\pi i\thi_2}| \\
&\le\frac{\de_0}{2R_0}\cdot\frac{1}{|\la (w_0)|}+R_02\sin(\pi(\thi_2-\thi_1))  \\ 
&<\frac{\de_0}{2R_0}\cdot\frac{1}{|\la (w_0)|}+2R_0\pi(\thi_2-\thi_1) \\
&<\frac{\de_0}{2R_0}\cdot\frac{1}{|\la (w_0)|}+2\pi R_0\cdot\frac{\de_0}{4\pi R_0}\cdot\frac{1}{|\la
(w_0)|}=\frac{\de_0}{2|\la (w_0)|}\bigg(\frac{1}{R_0}+1\bigg).
\end{align*}

%
So

\[
|\la (w_0)|\,|re^{2\pi i\thi}-r_1e^{2\pi i\thi_1}|<\frac{\de_0}{2}\bigg( \frac{1}{R_0}+1\bigg)<\de_0,
\]
because $R_0>1$, which implies (\ref{sec5eq8}). For $z$ with $|z|\leq R_1$ we have
\begin{align}
|f(z+a\la (w_0))-p(z)|\le&|f(z+a\la (w_0))
-p(z+\la (w_0)(re^{2\pi i\thi}-r_1e^{2\pi i\thi_1}))|\nonumber \\
&+|p(z+\la (w_0)(re^{2\pi i\thi}-r_1e^{2\pi i\thi_1}))-p(z)|.  \label{sec5eq9}
\end{align}
Previously, we proved that for every $|z|\le R_1$ we have $z+a\la (w_0)\in B_{w_0}$. Thus, by the definition of
$h$ and (\ref{sec5eq4}) we have
\begin{equation} \label{sec5eq10}
|f(z+a\la (w_0))-p(z+\la (w_0)(re^{2\pi i\thi}-r_1e^{2\pi i\thi_1}))|<\frac{1}{2s_1}.
\end{equation}
By (\ref{sec5eq8}) and (\ref{sec5eq3}) for $|z|\leq R_1$ we have
\begin{equation} \label{sec5eq11}
|p(z+\la (w_0)(re^{2\pi i\thi}-r_1e^{2\pi i\thi_1}))-p(z)|<\frac{1}{2s_1}.
\end{equation}
By (\ref{sec5eq9}), (\ref{sec5eq10}) and (\ref{sec5eq11}) we get
\[
\sup_{|z|\le R_1}|f(z+a\la (w_0))-p(z)|<\frac{1}{s_1}.
\]
So
\begin{equation} \label{sec5eq12}
\sup_{|z|\le k_1}|f(z+a\la (w_0))-p(z)|<\frac{1}{s_1}.
\end{equation}
Setting
$$
m_1:=\max \{ n\in \mathbb{N}: \lambda_n=\lambda (w) \,\,\, \textrm{for some} \,\,\, w\in P_m \} .
$$
we have that:

for every $a\in S$ there exists $w_0\in P_m$ such that $\lambda (w_0)=\lambda_n$ for some $n\in \mathbb{N}$ with
$n\leq m_1$ and (\ref{sec5eq12}) holds. Clearly the last implies that $f\in E(m_1,j_1,s_1,k_1)$, (\ref{sec5eq1})
holds and the proof of Lemma \ref{lem3} is complete. \qb\medskip

\section{Proof of Lemma \ref{lem1}}

By Mergelyan's theorem it easily follows that
\[
U:=\bigcap^\infty_{j=1}\bigcap^\infty_{s=1}\bigcap^\infty_{k=1}\bigcup^\infty_{m=1}
E(m,j,s,k)\subseteq\bigcap_{a\in S}HC(\{ T_{\la_na}\} ).
\]
We have to show the reverse inclusion. For every polynomial $p$ of one complex variable with coefficients in
$\Q+i\Q$ define the set
\[
\cu(p):=
\]
\[
\Big\{f\in\ch(\C)|\,\forall \,\, a\in S \,\, \exists \,\, (m_n) \subset \mathbb{N}: \,\forall  r>0\,\,
\dis\lim_{n\to +\infty}\sup_{|z|\leq  r}|f(z+\la_{m_n}a)-p(z)|=0  \Big\} .
\]
Let $p_j$, $j=1,2,\ld$ be an enumeration of all polynomials of one complex variable with coefficients in
$\Q+i\Q$. We see easily that

\begin{eqnarray}
\bigcap_{a\in S}HC(\{ T_{\la_na}\} )=\bigcap^\infty_{j=1}\cu(p_j).  \label{sec6eq1}
\end{eqnarray}
For $x>0$ and $n,j\in\N$ define the set

$$V(x,n,j):=$$
$$
\Big\{f\in\ch(\C)|\,\, \forall a\in S \,\,\exists \,\, m\in\N, m\le n \,\,\textrm{with}\,\, \dis\sup_{|z|\le
x}|f(z+\la_ma)-p_j(z)|<\frac{1}{x}\Big\}.
$$
We shall show that the following holds:
\begin{eqnarray}
\cu(p_j)\subseteq\bigcap_{x>0}\bigcup^\infty_{n=1}V(x,n,j).  \label{sec6eq2}
\end{eqnarray}
Let $f\in\ch(\C)$, $x_0>0$, $j_0,m_0\in\N$ and consider the set
\[
V_f(j_0,x_0,m_0):=\bigg\{a\in S|\,\,\sup_{|z|\le x_0}|f(z+\la_{m_0}a)-p_{j_0}(z)|<\frac{1}{x_0}\bigg\}.
\]
We first show that $V_f(j_0,x_0,m_0)$ is open in $S$. Let $a_0\in V_f(j_0,x_0,m_0)$ and take $(a_{\nu })$ a
sequence in $S$ such that $a_{\nu }\ra a_0$. We have
\begin{align}
\sup_{|z|\le x_0}|f(z+\la_{m_0}a_{\nu })-p_{j_0}(z)|\le&\sup_{|z|\le x_0}
|f(z+\la_{m_0}a_0)-p_{j_0}(z)| \nonumber \\
&+\sup_{|z|\le x_0}|f(z+\la_{m_0}a_{\nu })-f(z+\la_{m_0}a_0)|  \label{sec6eq3}
\end{align}
for every $\nu=1,2,\ld\,$.

The function $\f:S\times \overline{D(0,x_0)}\ra\C$ defined by $\f(a,z)=z+\la_{m_0}a$ is continuous, where the
set $S\times\overline{D(0,x_0)}$ is endowed with the product topology,
\begin{align*}
&\rho:(S\times\overline{D(0,x_0)})\times(S\times\overline{D(0,x_0)})\ra\R^+\;\;
\rho((\bi,z_1),(\ga,z_2))\\
&=\sqrt{|\bi-\ga|^2+|z_1-z_2|^2},\; \bi,\ga\in S, \;z_1,z_2\in \overline{D(0,x_0)}.
\end{align*}
Setting
\[
\e_0:=\frac{1}{x_0}-\sup_{|z|\le x_0}|f(z+\la_{m_0}a_0)-p_{j_0}(z)|,
\]
we observe that $\e_0>0$ since $a_0\in V_f(j_0,x_0,m_0)$. By the uniform continuity of $f\circ\f$ on
$S\times\overline{D(0,x_0)}$, there exists $\de_0>0$ such that for each $x,y\in S\times\overline{D(0,x_0)}$,
$\rho(x,y)<\de_0$ it holds $|(f\circ\f)(x)-(f\circ\f)(y)|<\e_0$. Since $a_{\nu}\ra a_0$, there exists $n_0\in\N$
such that $|a_{\nu}-a_0|<\de_0$ for each $\nu \in\N$, $\nu\ge n_0$. Now for every $z\in\overline{D(0,x_0)}$ and
$\nu\ge n_0$, $\nu\in\N$, we have:
\[
\rho((a_{\nu },z),(a_0,z))=\sqrt{|a_{\nu }-a_0|^2+|z-z|^2}=|a_{\nu }-a_0|<\de_0.
\]
So
\[
|(f\circ\f)(a_{\nu },z)-(f\circ\f)(a_0,z)|<\e_0 , \,\,\, \nu\geq n_0
\]
which in turn implies
\begin{equation}
\sup_{|z|\le x_0}|f(z+\la_{m_0}a_0)-p_{j_0}(z)| +\sup_{|z|\le x_0}|f(z+\la_{m_0}a_{\nu
})-f(z+\la_{m_0}a_0)|<\frac{1}{x_0} \label{sec6eq4}
\end{equation}
for $\nu \geq n_0$. In view of (\ref{sec6eq3}) and (\ref{sec6eq4}) there exists $n_0\in\N$ such that for every
$\nu \ge n_0$, $a_{\nu }\in V_f(j_0,x_0,m_0)$. From the last we conclude that the set $V_f(j_0,x_0,m_0)$ is
open.

Thus, for every $f\in\ch(\C),j,m\in\N$ and every $x>0$ the set $V_f(j,x,m)$ is open in $S$.

Take $g\in\cu(p_{j_0})$. Then for each $a\in S$ there exists a subsequence $(\la_{m_n(a)})$ of $(\la_n)$ (that
depends on $a$) such that for every $r>0$
\[
\sup_{|z|\leq r}|g(z+\la_{m_n(a)}a)-p_{j_0}(z)|\ra0 \ \ \text{as} \ \ n\ra+\infty.
\]
In particular we get
\[
\sup_{|z|\le x_0}|g(z+\la_{m_n(a)}a)-p_{j_0}(z)|\ra0 \ \ \text{as} \ \ n\ra+\infty.
\]
Thus, for $\e=\frac{1}{x_0}$ we have that for every $a\in S$ there exists $n_a\in\N$ (that depends on $a$) so
that for each $n\ge n_a$, $n\in\N$, it holds
\[
\sup_{|z|\le x_0}|g(z+\la_{m_n(a)}a)-p_{j_0}(z)|<\frac{1}{x_0}.
\]
Therefore, the set
\[
\cn(j_0,x_0,g):=\bigg\{n\in\N|\exists\;a\in S:\sup_{|z|\le x_0}|g(z+\la_na) -p_{j_0}(z)|<\frac{1}{x_0}\bigg\}
\]
is non-empty. It is obvious by the above definitions that
\begin{eqnarray}
V_g(j_0,x_0,m)\subset S \ \ \text{for each} \ \ m\in\N.  \label{sec6eq5}
\end{eqnarray}
Let some $a\in S$. Then there exists $n\in\cn(j_0,x_0,g)$ such that $a\in V_g(j_0,x_0,n)$. Hence we get
\begin{eqnarray}
S\subseteq\bigcup_{n\in\cn(j_0,x_0,g)}V_g(j_0,x_0,n).  \label{sec6eq6}
\end{eqnarray}
Now, (\ref{sec6eq5}) and (\ref{sec6eq6}) imply
\[
S=\bigcup_{n\in\cn(j_0,x_0,g)}V_g(j_0,x_0,n),
\]
so the family $V_g(j_0,x_0,n)$, $n\in\cn(j_0,x_0,g)$ is an open covering of $S$. Since $S$ is a compact set
there exists a finite subset $A\subset\cn(j_0,x_0,g)$, $A=\{{\nu}_1,{\nu}_2,\ld,{\nu}_{m_0}\}$ such that
$S=\bigcup\limits^{m_0}_{n=1}V_g(j_0,x_0,{\nu}_n)$. Let $\el_0:=\max A$. Then for each $a\in S$, there exists
$n\in\N$, $n\le\el_0$ such that
\[
\sup_{|z|\le x_0}|g(z+\la_na)-p_{j_0}(z)|<\frac{1}{x_0}.
\]
It follows that $\cu(p_{j_0}) \subset V(x_0,\el_0,j_0)$ for arbitrary $x_0>0$, from which we get
\begin{eqnarray}
\cu(p_{j_0}) \subset \bigcap_{x>0}\bigcup^\infty_{n=1}V(x,n,j_0).  \label{sec6eq7}
\end{eqnarray}
Thus (\ref{sec6eq2}) holds for every $j=1,2,\ld\,.$ It is obvious that
\begin{eqnarray}
\bigcap_{x>0}\bigcup^\infty_{n=1}V(x,n,j_0)\subset\bigcap^\infty_{m=1} \bigcup^\infty_{n=1}V(m,n,j_0).
\label{sec6eq8}
\end{eqnarray}
By (\ref{sec6eq2}) and (\ref{sec6eq8}) we get
\begin{eqnarray}
\cu(p_j)\subset\bigcap^\infty_{m=1}\bigcup^\infty_{n=1}V(m,n,j) \ \ \text{for every} \ \ j=1,2,\ldots .
\label{sec6eq9}
\end{eqnarray}
So
\begin{eqnarray}
\bigcap^\infty_{j=1}\cu(p_j)\subset\bigcap^\infty_{j=1}\bigcap^\infty_{m=1} \bigcup^\infty_{n=1}V(m,n,j).
\label{sec6eq10}
\end{eqnarray}
By (\ref{sec6eq1}) and (\ref{sec6eq10}) we have
\begin{eqnarray}
\bigcap_{a\in S}HC(\{ T_{\la_na}\} ) \subset\bigcap^\infty_{j=1}\bigcap^\infty_{m=1}\bigcup^\infty_{n=1}
V(m,n,j), \label{sec6eq11}
\end{eqnarray}
and now it is plain that
\begin{eqnarray}
U\subset\bigcap^\infty_{j=1}\bigcap^\infty_{m=1}\bigcup^\infty_{n=1}V(m,n,j). \label{sec6eq12}
\end{eqnarray}
We consider the following families of sets
\[
\cde_1:=\bigg\{\bigcup^\infty_{n=1}V(m,n,j),m,j\in\N\bigg\} \ \ \text{and}
\]
\[
\cde_2:=\bigg\{\bigcup^\infty_{m=1}E(s,j,k,m),s,j,k\in\N\bigg\}.
\]
Clearly $\cde_1\subseteq\cde_2$. Thus
\begin{eqnarray}
\bigcap_{E\in \cde_2}E\subset\bigcap_{V\in \cde_1}V.  \label{sec6eq13}
\end{eqnarray}
Let $E=E(s,j,k,m)$ for some $s,j,k,m\in\N$. Then $V(\el,m,j)\subset E$ for $\el=\max\{s,k\}$. Hence, for every
$E\in\cde_2$ there exists $\Ga\in\cde_1$ such that $\Ga\subset E$. If we set
$\widetilde{\cde}=\{\Ga\in\cde_1|\exists\;E\in\cde_2:\Ga\subset E\}$, it follows that
$\bigcap\limits_{\Ga\in\widetilde{\cde}}\Ga\subset \bigcap\limits_{E\in\cde_2}E$. But then
\begin{eqnarray}
\bigcap_{V\in \cde_1}V \subset\bigcap_{\Ga\in\widetilde{\cde}}\Ga\subset\bigcap_{E\in \cde_2}E. \label{sec6eq14}
\end{eqnarray}
By (\ref{sec6eq12}), (\ref{sec6eq13}) and (\ref{sec6eq14}) we have
\begin{eqnarray}
U=\bigcap^\infty_{j=1}\bigcap^\infty_{m=1}\bigcup^\infty_{n=1}V(m,n,j). \label{sec6eq15}
\end{eqnarray}
Now (\ref{sec6eq11}) and (\ref{sec6eq15}) yield
\begin{eqnarray}
\bigcap_{a\in S}HC(\{ T_{\la_na}\} )\subset U
\end{eqnarray}
and the proof of lemma \ref{lem1} is complete. \qb

The above lemma holds with the same proof for every compact subset $K\subseteq\C\setminus \{ 0\}$ instead of $S$
and for every sequence of non-zero complex numbers $(\la_n)$ such that $\la_n\ra\infty$ as $n\ra+\infty$.

\section{Final step of the proof of Theorem \ref{thm1} }

To conclude the proof of Theorem \ref{thm1} we need the following three elementary lemmas.

\begin{lem}\label{sec7lem1}
Let $(\la_n)$ be a sequence of non-zero complex numbers such that $\la_n\ra\infty$ as $n\ra+\infty$. Suppose
that $\underset{n\ra+\infty}{\lim\sup}\Big|\dfrac{\la_{n+1}}{\la_n}\Big|=1$. Then for any fixed positive numbers
$M_1,M_2$ there exists a subsequence $(\mi_n)$ of $(\la_n)$ with the following properties:
\begin{enumerate}
\item[1)] $|\mi_{n+1}|-|\mi_n|>M_1$, for every $n=1,2,\ld$,
\item[2)] $\Big|\dfrac{\mi_{n+1}}{\mi_n}\Big|\ra1$ as $n\ra+\infty$,
\item[3)] $\underset{n\ra+\infty}{\lim\inf}\Big(n\Big(\Big|\dfrac{\mi_{n+1}}{\mi_n}\Big|
    -1\Big)\Big)>M_2$.
\end{enumerate}
\end{lem}
\begin{proof}
We prove this lemma in three steps:

{\bf Step 1.} We construct a subsequence $(\thi_n)$ of $(\la_n)$ such that $(|\thi_n|)$ is strictly increasing
and $\Big|\dfrac{\thi_{n+1}}{\thi_n}\Big|\ra1$ as $n\ra+\infty$.

{\bf Step 2.} We construct a subsequence $(k_n)$ of $(\thi_n)$ such that $|k_{k+1}|-|k_n|>M_1$ $\fa n=1,2,\ld$
and $\Big|\dfrac{k_{n+1}}{k_n}\Big|\ra1$ as $n\ra+\infty$.

{\bf Step 3.} Finally, we construct a subsequence $(\mi_n)$ of $(k_n)$ which has the three properties 1), 2) and
3) of the lemma.

{\bf Proof of Step 1.}

We set $\thi_1:=\la_1$. Let $n_1\ge2$ be the smallest natural number such that $|\la_{n_1}|>|\la_1|$. Define
$\thi_2:=\la_{n_1}$. Suppose now that we have constructed inductively the numbers
$\la_{n_1},\la_{n_2},\ld,\la_{n_k}$ for some $k\ge 2$, where $|\la_{n_{i+1}}|>|\la_{n_i}|$ and $n_{i+1}$ is the
smallest natural number such that $n_{i+1}>n_i$ and $|\la_{n_{i+1}}|>|\la_{n_i}|$ for every $i=1,2,\ld,k-1$. Set
$\thi_{i+1}=\la_{n_i}$ for $i=1,\ld,k$. Next we consider the number $\la_{n_{k+1}}$, where $n_{k+1}$ is the
smallest natural number with $n_{k+1}\ge n_k+1$ and such that $|\la_{n_{k+1}}|>|\la_{n_k}|$ and we set
$\thi_{k+2}=\la_{n_{k+1}}$.

So, we have constructed a subsequence $(\thi_n)$ of $(\la_n)$ such that the sequence $(|\thi_n|)$ is strictly
increasing. For every $k\in \mathbb{N}$ we have
\[
1<\bigg|\frac{\la_{n_{k+1}}}{\la_{n_k}}\bigg|\le\bigg|\frac{\la_{n_{k+1}}} {\la_{n_{k+1}-1}}\bigg|
\]
and by our assumptions on $(\lambda_n)$ we conclude that $|\thi_{n+1}|/|\thi_n|\to 1$.

{\bf Proof of Step 2.}

Now we construct a subsequence of $(\thi_n)$ as follows. We set $k_1:=\thi_1$. Let $v_1$ be the smallest natural
number such that $v_1\ge 2$ and $|\thi_{v_1}|>|k_1|+M_1$. Set now $k_2:=\thi_{v_1}$. Suppose that we have
constructed inductively the numbers $\thi_1,\thi_{v_1},\ld,\thi_{v_m}$ for some $m\ge 2$, where $v_{i+1}$ is the
smallest natural number such that $|\thi_{v_{i+1}}|>|\thi_{v_i}|+M_1$ and $v_{i+1}>v_i$ for each $i=1,\ld,m-1$.
Then set $k_{i+1}=\thi_{v_i}$ for $i=1,\ld,m$. Next we consider the smallest natural number $v_{m+1}\ge v_m+1$
such that $|\thi_{v_{m+1}}|>|\thi_{v_m}|+M_1$ and we set $k_{m+2}=\thi_{v_{m+1}}$.

Therefore we have constructed a subsequence $(k_n)$ of $(\thi_n)$ where $|k_{n+1}|> |k_n|+M_1$ for each
$n=1,2,\ldots $. For every $m=1,2,\ldots $ it holds that $|\thi_{v_m}|\le|\thi_{v_{m+1}-1}|\leq
|\thi_{v_m}|+M_1$, which implies
\begin{eqnarray}
1\le \bigg|\frac{\thi_{v_{m+1}-1}}{\thi_{v_m}}\bigg| \leq 1+\frac{M_1}{|\thi_{v_m}|}. \label{sec7eq3}
\end{eqnarray}
On the other hand we have
\begin{eqnarray}
1<\bigg|\frac{\thi_{v_{m+1}}}{\thi_{v_{m+1}-1}}\bigg|, \,\,\,\, m=1,2,\ldots  \label{sec7eq4}
\end{eqnarray}
and
\begin{eqnarray}
\lim_{n\ra+\infty}\bigg|\frac{\thi_{n+1}}{\thi_n}\bigg|=1.  \label{sec7eq5}
\end{eqnarray}
By (\ref{sec7eq3}), (\ref{sec7eq4}) and (\ref{sec7eq5}) we conclude that $|k_{n+1}|/|k_n|\to 1$ as
$n\ra+\infty$.

{\bf Proof of Step 3.}

We construct inductively a subsequence $(\mi_n)$ of $(k_n)$ as follows. Set $\mi_1:=k_1$. Let $\si_1$ be the
smallest natural number such that $\si_1\ge2$ and $|k_{\si_1}|>|k_1|\Big(1+\dfrac{M_2}{1}\Big)$ and then define
$\mi_2:=k_{\si_1} $. After, let $\si_2$ be the smallest natural number such that $\si_2\geq \si_1+1$ ,
$k_{\si_2}\ge\mi_2+1$ and $|k_{\si_2}|>|k_{\si_1}|\cdot\Big(1+\dfrac{M_2}{2}\Big)$ and define
$\mi_3:=k_{\si_2}$. In this way, we construct inductively a subsequence $(\mi_n)$ of $(k_n)$ such that for every
$n=2,3,\ldots $ the natural number $\si_n$ is the smallest with the following properties: $k_{\si_n}\ge\mi_n+1$,
$\si_n\geq \si_{n-1}+1$,
\begin{eqnarray}
|\mi_{n+1}|\ge|\mi_n|\bigg(1+\frac{M_2}{n}\bigg),   \label{sec7eq6}
\end{eqnarray}
and $\mi_{n+1}=k_{\si_n}$.

As a consequence of the above construction we get
\begin{eqnarray}
1\le\bigg|\frac{k_{\si_{n+1}-1}}{k_{\si_n}}\bigg|<1+\frac{M_2}{n+1}, \,\,\, n=1,2,\ldots \label{sec7eq7}
\end{eqnarray}
\begin{eqnarray}
1<\bigg|\frac{k_{\si_{n+1}}}{k_{\si_{n+1}-1}}\bigg|, \,\,\, n=1,2,\ldots \label{sec7eq8}
\end{eqnarray}
\begin{eqnarray}
\bigg|\frac{k_{n+1}}{k_n}\bigg|\ra1 \ \ \text{as} \ \ n\ra+\infty.  \label{sec7eq9}
\end{eqnarray}
By  (\ref{sec7eq7}), (\ref{sec7eq8}) and (\ref{sec7eq9}) we conclude that $|\mi_{n+1}|/|\mi_n|\to 1$ as
$n\ra+\infty$ and the sequence $(\mi_n)$ has all the desired properties. This completes the proof the lemma.
\end{proof}

\begin{lem}\label{sec7lem2}
Let $(\la_n)$ be a sequence of non-zero complex numbers such that $\la_n\ra\infty$ as $n\ra+\infty$. Suppose
that
\[
\underset{n\ra+\infty}{\lim\sup}\bigg|\frac{\la_{n+1}}{\la_n}\bigg|\le1+\e
\]
for some $\e>0$. Then for every pair $(M_1,M_2)$ of positive numbers there exists a subsequence $(\mi_n)$ of
$(\la_n)$ with the following properties:
\begin{enumerate}
\item[1)] $|\mi_1|>M_1$,
\item[2)] $|\mi_{n+1}|-|\mi_n|>M_1$, $n=1,2,\ld$,
\item[3)] $\underset{n\ra+\infty}{\lim\sup}\Big|\dfrac{\mi_{n+1}}{\mi_n}\Big|\le1+\e$,
\item[4)] $\underset{n\ra+\infty}{\lim\inf}\Big(n\Big(\Big|\dfrac{\mi_{n+1}}{\mi_n}\Big|
    -1\Big)\Big)>M_2$.
\end{enumerate}
\end{lem}
\begin{proof}
The proof is almost identical to the proof of Lemma \ref{sec7lem1}. The only difference is that, whenever
needed, instead of $\underset{n\ra+\infty}{\lim\sup}\Big|\dfrac{\la_{n+1}}{\la_n}\Big|=1$ we use
$\underset{n\ra+\infty}{\lim\sup}\Big|\dfrac{\la_{n+1}}{\la_n}\Big|\le1+\e$.
\end{proof}

By Lemma \ref{sec7lem2} along with elementary considerations we obtain the following lemma, whose proof is left
to the interested reader.
\begin{lem} \label{sec7lem3}
Let $\Lambda:=(\la_n)$ be a fixed sequence of non-zero complex numbers such that $\la_n\ra\infty$ as
$n\ra+\infty$. Then $i(\Lambda )=1$ if and only if for every positive number $\sigma_j$, $j=1,2,3,4,5$ and
positive integers $m_0$, $k_0$ with $m_0\geq [\sigma_1]+1$, $k_0\geq [\sigma_3]+1$ there exist a subsequence
$(\mu_n)$ of $(\la_n)$ and a positive integer $n_0$ such that for every $n\geq n_0$ the following five
properties hold:
\begin{enumerate}
\item[1)] $|\mu_n|\cdot\ssum^{m_0-1}_{k=0}\dfrac{1}{|\mu_{n+k}|}>\sigma_1$
\item[2)] $|\mu_{n+1}|-|\mu_n|>\sigma_2$
\item[3)]  $|\mu_n|\cdot\ssum^{k_0}_{i=1}\dfrac{1}{|\mu_{n+im_0-1}|}>\sigma_3$
\item[4)] $n\bigg(\bigg|\dfrac{\mu_{n+1}}{\mu_n}\bigg|-1\bigg)>\sigma_4$
\item[5)] $\dfrac{n}{|\mu_n|}<\sigma_5 .$
\end{enumerate}
\end{lem}

{\bf Proof of Theorem \ref{thm1}}.

A careful inspection of the proof of Proposition \ref{prop1} shows that the conclusion of Proposition
\ref{prop1} holds whenever the sequence $(\la_n)$ satisfies the properties $(2.1)-(2.8)$ in subsection
\ref{good}. Fix a sequence $(\la_n)$ of non-zero complex numbers such that $\la_n\ra\infty$ as $n\ra+\infty$. In
view of Lemmas \ref{sec7lem2}, \ref{sec7lem3}, it is easy to show that there exists a subsequence $(\mu_n)$ of
$(\lambda_n)$ which satisfies properties $(2.1)-(2.8)$. Therefore $$\bigcap_{a\in S}HC(\{ T_{\mu_na}\} )$$ is
$G_\de$ and dense subset of $(\ch(\C),\ct_u)$ and since $$\bigcap_{a\in S}HC(\{ T_{\mu_na}\} ) \subset
\bigcap_{a\in S}HC(\{ T_{\la_na}\} )$$ it readily follows that $$\bigcap_{a\in S}HC(\{ T_{\la_na}\} )$$ is
$G_\de$ and dense subset of $(\ch(\C),\ct_u)$. Then, applying once more Baire's category theorem, see the
discussion after the statement of Theorem \ref{thm2}, we conclude the proof of Theorem \ref{thm1}.
\section{Examples of sequences $\Lambda:=(\la_n)$ with $i(\Lambda )=1$}
Let $\Lambda =(\la_n)$ be a sequence of non-zero complex numbers. Define the set
$$
\mathcal{B}(\Lambda ):= \Big\{ a\in [0,+\infty ]| \,\, \exists (\mu_n) \subset \Lambda \,\,\,\textrm{with}\,\,\, a=\limsup_{n} \Big|
\frac{\mu_{n+1}}{\mu_n }\Big| \Big\} .
$$
Observe now that, by definition, $i(\Lambda )=\inf\mathcal{B}(\Lambda )$ and whenever $\la_n\to \infty$ then
$\mathcal{B}(\Lambda ) \subset [1,+\infty]$. We shall present four distinct classes of sequences $\Lambda:=
(\la_n)$ satisfying the property $i(\Lambda )=1$ in order to illustrate our main result, Theorem \ref{thm1}.\\

1) \textit{Examples with $\la_n\to \infty$ and $|\la_{n+1}|/|\la_n|\to 1$}.\\

A sample of sequences satisfying the previously mentioned properties is: $n$, $n^2$, $p(n)$ where $p$ is a
non-constant complex polynomial, $\log n$, $n^{\beta} \log n$, $\beta>0$, $n^{\gamma}/\log (n+1)$, $\gamma
>2$ etc. Of course, one can assign to each term of the above sequences fixed unimodular numbers with arbitrary
arguments and still the desired properties are satisfied i.e. $e^{i\theta_n}n^2$, $e^{i\theta_n }\log n$ for
$\theta_n \in \mathbb{R}$, etc.

A more interesting example is the sequence $e^{n^c}$, for $0<c<1$, which has super-polynomial growth. Observe
that the case $c=1$, is a borderline for the validity of Theorem \ref{thm1}. Indeed, as we already mentioned in
the Introduction,
$$
\bigcap_{a\in \{ z: |z|=1 \} }HC(\{ T_{e^na}\} )=\emptyset ,
$$
by the main result in \cite{8}.

A last family of sequences, satisfying the above properties, we would like to mention is the following:
$e^{\frac{n}{\log n}}$, $e^{\frac{n}{\log \log n}}$, etc. Note that such sequences
grow faster than any sequence of the form $e^{n^c}$, $0<c<1$. \\

2) \textit{Examples with $\la_n\to \infty$, the limit $\lim_{n\to +\infty}|\la_{n+1}/\la_n|$ does not exist, but
$\limsup_{n}|\la_{n+1}/\la_n|=1$} .\\

There is a plethora of sequences exhibiting such a behavior. For instance, set $\la_1=1$. We shall define the
sequence $(\la_n)$ inductively according to the following rule. If for some $k\in \mathbb{N}$ there exists $n\in
\mathbb{N}$ such that $\la_k=n^2$ then define
$$\la_{k+i}:=n+i-1 \,\,\, \textrm{for every}\,\,\, i=1,2,\ldots , n^2+n+2 .$$
It is easy to show that the sequence $(\la_n)$ has the desired properties.\\

3) \textit{Examples with $\la_n\to \infty$, $\limsup_{n}|\la_{n+1}/\la_n|>1$ and
$\limsup_{n}|\mu_{n+1}/\mu_n|=1$ for some subsequence $(\mu_n)$ of $(\la_n)$}.\\

Take $\la_{2n+1}=n$, $\la_{2n}=2^n$ for $n=1,2,\ldots $ or more general fix a sequence of positive numbers
$\gamma_n$ satisfying $\gamma_n\to \infty$, $\gamma_{n+1}/\gamma_n\to 1$, consider a strictly increasing
sequence $(m_n)$ of positive integers with $m_n>n$ for every $n=1,2,\ldots $ and then define
$\la_{m_n}=\gamma_n$ and on the set $\{ \rho_1 < \rho_2< \cdots \} := \mathbb{N}\setminus \{ m_n: n=1,2,\ldots
\}$ define $\lambda_{\rho_n}$ to be any positive number such that $\lambda_{\rho_n}\to +\infty$ and
$\lambda_{\rho_{n+1}}/\lambda_{\rho_n} \to c$ for some $c\in (1,+\infty ]$.\\

4) \textit{Examples with $\la_n\to \infty$, $\inf\mathcal{B}(\Lambda ) \notin \mathcal{B}(\Lambda )$ and $i(\Lambda)=1$}.\\

In all the above examples we have that $\inf\mathcal{B}(\Lambda ) \in \mathcal{B}(\Lambda )$. This means that
the above infimum becomes minimum. We shall now differentiate from this situation by exhibiting examples of
$\Lambda=(\la_n)$ such that $\la_n\to \infty$, $i(\Lambda)=1$ and for every subsequence $(\mu_n)$ of $(\la_n)$
we have $\limsup_{n}|\mu_{n+1}/\mu_n|>1$. To produce such an example is not an easy task, as it requires a
considerable amount of work, though elementary, concerning a particular representation of positive integers
involving powers of $10$. Therefore, we omit the details and we just state the following lemma without proof.

\begin{lem} \label{sec8lem}
For every positive integer $n\geq 11$ there exists a unique triple $(\nu ,k,j)$ with $\nu \in
\mathbb{N}\setminus \{ 1\} $, $k\in \{ 1,2,\ldots ,\nu \}$, $j\in \{ 1,2,\ldots ,10^k \}$ such that
$$n=\frac{10}{9} \Big( \frac{10}{9} (10^{\nu -1}-1)-\nu+10^{k-1} \Big)+j.$$
\end{lem}
Define now the sequence $(\la_n)$ by
$$ \la_n=\Big( 1+\frac{1}{k} \Big)^{(\nu-k+1)10^k+j} \,\,\, \textrm{for} \,\,\, n\geq 11 ,$$ where for every given
positive integer $n$ with $n\geq 11$, the numbers $\nu $, $k$, $j$ are uniquely determined by Lemma
\ref{sec8lem}. It turns out, after a lengthy argument, that the sequence $(\la_n)$ has the desired properties.\\

{\bf Acknowledgements:} I am grateful to Professor Stephen Gardiner, who gave me the chance to deal with the
present work, for his useful remarks/comments and all the help he offered me during this project. I am also
grateful to Dr. Myrto Manolaki, who read this work, for drawing beautiful designs of the disks constructed in
section 4 and for her interest in this work. Finally, I am grateful to Professor George Costakis for showing
great interest in this work.

\end{document}